\documentclass[
final
]{dmtcs-episciences}

\usepackage[utf8]{inputenc}
\usepackage{subfigure}

\usepackage[round]{natbib}

\usepackage{tikz,color,url}
\usepackage{amsmath,amssymb,latexsym}
\usetikzlibrary{arrows}

\newtheorem{theorem}{Theorem}[section]
\newtheorem{definition}[theorem]{Definition}
\newtheorem{lemma}[theorem]{Lemma}
\newtheorem{proposition}[theorem]{Proposition}
\newtheorem{observation}[theorem]{Observation}

\newtheorem{corollary}[theorem]{Corollary}
\newtheorem{remark}[theorem]{Remark}

\newtheorem{problem}[theorem]{Open Problem}

\newcommand{\gsmb}{\gamma_{\rm SMB}}
\newcommand{\gmb}{\gamma_{\rm MB}}

\newcommand{\po}{{\cal P}_{\rm o}}
\newcommand{\pco}{{\cal P}_{\rm co}}

\newcommand{\cH}{{\cal H}}
\newcommand{\cF}{{\cal F}}
\newcommand{\cD}{{\cal D}}
\newcommand{\cS}{{\cal S}}
\newcommand{\cP}{{\cal P}}
\newcommand{\cA}{{\cal A}}
\newcommand{\cB}{{\cal B}}
\newcommand{\cT}{{\cal T}}

\newcommand{\w}{{\rm w}}

\author[Csilla Bujt\'as et al.]{Csilla Bujt\'as\affiliationmark{1,2,}\thanks{Supported by the Slovenian Research Agency (ARIS) under the grant P1-0297.}
  \and Pakanun Dokyeesun\affiliationmark{1,}\thanks{Supported by PhD scholarship granted by The Institute for the Promotion of Teaching Science and Technology (IPST), Thailand.}
  \and Sandi Klav\v zar\affiliationmark{1,2,3,}\thanks{Supported by the Slovenian Research Agency (ARIS) under the grants P1-0297, J1-2452, and N1-0285.}}
\title[Maker-Breaker domination game on trees when Staller wins]{Maker-Breaker domination game on trees when Staller wins}
\affiliation{
  Faculty of Mathematics and Physics, University of Ljubljana, Slovenia\\
  Institute of Mathematics, Physics and Mechanics, Ljubljana, Slovenia\\
  Faculty of Natural Sciences and Mathematics, University of Maribor, Slovenia}
\keywords{ domination game; Maker-Breaker game; Maker-Breaker domination game; hypergraph; tree; subdivided star; caterpillar}
\begin{document}
\publicationdata{vol. 25:2 }{2023}{12}{10.46298/dmtcs.10515}{2022-12-20; 2022-12-20; 2023-04-27}{2023-07-13}
\maketitle
\begin{abstract}
  In the Maker-Breaker domination game played on a graph $G$, Dominator's goal is to select a dominating set and Staller's goal is to claim a closed neighborhood of some vertex. We study the cases when Staller can win the game. If Dominator (resp., Staller) starts the game, then $\gamma_{\rm SMB}(G)$ (resp., $\gamma_{\rm SMB}'(G)$) denotes the minimum number of moves Staller needs to win. For every positive integer $k$, trees $T$ with $\gamma_{\rm SMB}'(T)=k$ are characterized and a general upper bound on $\gamma_{\rm SMB}'$ is proved. Let $S = S(n_1,\dots, n_\ell)$ be the subdivided star obtained from the star with $\ell$ edges by subdividing its edges $n_1-1, \ldots, n_\ell-1$ times, respectively. Then $\gamma_{\rm SMB}'(S)$ is determined in all the cases except when $\ell\ge 4$ and each $n_i$ is even. The simplest formula is obtained when there  are at least two odd $n_i$s. If $n_1$ and $n_2$ are the two smallest such numbers, then $\gamma_{\rm SMB}'(S(n_1,\dots, n_\ell))=\lceil \log_2(n_1+n_2+1)\rceil$. For caterpillars, exact formulas for $\gamma_{\rm SMB}$ and for $\gamma_{\rm SMB}'$ are established.
\end{abstract}

\section{Introduction}
\label{sec:in}
The Maker-Breaker game was introduced in~\cite{erdos_combinatorial_1973}. The game is played on an arbitrary hypergraph by two players who alternately select a non-played vertex of the hypergraph. One player, named Maker, wants to occupy all the vertices of some hyperedge, while the other player, named Breaker, tries to prevent him from doing it. If the first situation happens, then Maker is declared as the winner of the game, otherwise Breaker wins. The game, either in its general form, or in different special cases, was investigated a lot by now, see the book~\cite{hefetz_positional_2014}. For related recent developments see~\cite{clemens_maker-breaker_2021, day_makerbreaker_2021, nicholas_day_maker-breaker_2021, glazik_new_2022, kang_makerbreaker_2021, stojakovic_hamiltonian_2021} and references therein. The Maker-Breaker games have been recently studied also on digraphs in~\cite{frieze_maker_2021}.

The Maker-Breaker domination game, {\em MBD game} for short, was introduced by Duch\^{e}ne, Gledel, Parreau, and Renault in~\cite{duchene_makerbreaker_2020}. Among other results they proved that deciding the winner of the MBD game is PSPACE-complete in general, and showed that the problem can be solved efficiently on trees. Following the tradition of the theory of the domination game, the two players in this version of the Maker-Breaker game are called Dominator and Staller. This game was introduced in~\cite{bresar_domination_2010}, its state of the art till 2021 summarized in~\cite{bresar_domination_2021}, and is still being investigated, see e.g.~\cite{bujtas_12-conjectures_2022}. We also add that the Maker-Breaker total domination game was introduced in~\cite{gledel_makerbreaker_2020} and further investigated in~\cite{forcan_maker-breaker_2022}. 

The MBD game played on a graph $G=(V(G), E(G))$ can be considered as the Maker-Breaker game played on the hypergraph $\cD_G$ whose hyperedges are the minimal dominating sets of $G$. In this case, Dominator is Maker, and Staller is Breaker. Moreover, this game can also be considered as the Maker-Breaker game played on the closed neighborhood hypergraph $\cH_G$ of $G$, where the hyperedges are the closed neighborhoods of the vertices of $G$. Now Dominator is Breaker and Staller is Maker. 

A MBD game is called {\em D-game} (resp., {\em S-game}) if Dominator (resp., Staller) is the first to play a vertex. Suppose that Dominator has a winning strategy in the D-game. Then the {\em Maker-Breaker domination number}, $\gmb(G)$, of $G$ is the (minimum) number of moves of Dominator to win the game when both players play optimally. The corresponding invariant for the S-game is denoted by $\gmb'(G)$. These concepts were introduced in~\cite{gledel_makerbreaker_2019} and further studied on prisms in~\cite{forcan_maker--breaker_2023}. Clearly, the problem of determining the Maker-Breaker domination number is interesting when Dominator wins the game. For the situations in which Staller is the winner, the \emph{Staller-Maker-Breaker domination number} (\emph{SMBD-number} for short), $\gsmb(G)$, of $G$, is the (minimum) number of Staller's moves she needs to win the D-game if both players play optimally. If Staller has no winning strategy in the D-game, we set $\gsmb(G)=\infty$. For the S-game, the corresponding invariant is defined analogously and denoted by~$\gsmb'(G)$. These two invariants are from~\cite{bujtas_fast_2022}. 

In this paper we proceed the investigation of the SMBD-numbers. In~\cite{gledel_makerbreaker_2019}, exact formulas for $\gmb(T)$ and $\gmb'(T)$ were proved if $T$ is a tree. To determine $\gsmb(T)$ and $\gsmb'(T)$ turns out to be much more involved. In the main result of Section~\ref{sec:trees} we characterize, for every positive integer $k$, the class of trees $T$ with $\gsmb'(T)=k$. In the subsequent section we study subdivided stars. Since $\gsmb(T) = \infty$ holds for each subdivided star $T$, we focus on $\gsmb'(T)$ and determine the exact value in all the cases except when the star has at least four edges, and each edge is subdivided an odd number of times. For the latter case we prove a sharp upper bound on $\gsmb'$. In Section~\ref{sec:caterpillars} we determine $\gsmb(T)$ and $\gsmb'(T)$ for an arbitrary caterpillar $T$.    

\section{Preliminaries}

\subsection{Definitions}

In a simple graph $G=(V(G), E(G))$, the \emph{open neighborhood} $N_G(v)$ of a vertex $v \in V(G)$ is the set of vertices being adjacent to $v$, while $N_G[v]=N_G(v) \cup \{v\}$ is the \emph{closed neighborhood} of $v$. If $X\subseteq V(G)$, then $N_G[X] = \bigcup_{x\in X} N_G[x]$. The {\em degree} of $v$ is $\deg_G(v)= |N_G(v)|$. A \textit{leaf} is a vertex $v$ with  $\deg_G(v)=1$. A vertex $u \in V(G)$ is a \textit{support vertex}, if $N_G(u)$ contains a leaf. Moreover, $u$ is a \textit{strong support vertex} if it is adjacent to more than one leaf. Otherwise, it is a \textit{weak support vertex}.

A  set $S \subseteq V(G)$ is a \emph{dominating set} in $G$ if $N_G[S] =V(G)$. 
A dominating set $S$ is \emph{minimal} if there is no dominating set among the proper subsets of $S$. 

The {\em residual graph $R(G)$} of a graph $G$ is obtained from $G$ by iteratively removing a pendant path $P_2$ until no such path is present. By a pendant $P_2$ we mean  a $P_2$-component or a $P_2$ attached to the graph by an edge. In the latter case, when a pendant $P_2$ is removed, exactly two vertices and two edges are deleted, whilst in the former case we obtain the empty graph.
If $T$ is a tree, then $R(T)$ is either the empty graph, or 
$R(T) = K_1$, or each support vertex of $R(T)$ is of degree at least $3$. It was proved in~\cite{duchene_makerbreaker_2020} that the winner is the same in $T$ and in $R(T)$. 

A \emph{hypergraph} $\cH=(V(\cH), E(\cH))$ is a set system over the vertex set $V(\cH)$. The (hyper)edge set $E(\cH)$ might contain subsets of $V(\cH)$ of any cardinality that is, $E(\cH) \subseteq 2^{V(\cH)}$. If $|e|=2$ holds for every hyperedge $e \in E(\cH)$, the hypergraph $\cH$ corresponds to a graph. Thus, most of the basic definitions related to graphs can be generalized to hypergraphs (see \cite{berge_hypergraphs_1989}). For example, if $\cH_1$ and $\cH_2$ are two hypergraphs, we say that  $\cH_1$ is a \emph{subhypergraph} of $\cH_2$, if $V(\cH_1) \subseteq V(\cH_2)$ and $E(\cH_1) \subseteq E(\cH_2)$.  

While studying Maker-Breaker games, we will refer to two operators on hypergraphs. Given a subset $X$ of the vertex set $V(\cH)$, the hypergraph $\cH-X$ is obtained from $\cH$ by removing the vertices in $X$ and all incident edges. That is, 
$V(\cH-X)=V(\cH)\setminus X$ and
$$ E(\cH-X)=\{e: e \in E(\cH) \enskip  {\rm and }\enskip e \subseteq V(\cH)\setminus X\}.$$
The second operator, named \emph{shrinking}, creates a hypergraph  denoted by $\cH\mid X$. Here, the vertices in $X$ are deleted again from $\cH$ but the incident hyperedges are just `shrinked' instead of being deleted. Formally, $V(\cH\mid X)=V(\cH)\setminus X$, and 
$$E(\cH\mid X)=\{e \setminus X: e \in E(\cH)\}.$$
If $X=\{v\}$, we may write $\cH-v$ and $\cH \mid v$ instead of $\cH- \{v\}$ and $\cH\mid \{v\}$, respectively.

When a Maker-Breaker game is played on a hypergraph $\cH$, we say that the hyperedges of $\cH$ are the \emph{winning sets} and Maker wins the game with his $i^{\rm th}$ move if he claims a winning set with this move. A \emph{winning strategy} of Maker is a strategy which ensures that he wins on $\cH$ no matter what strategy is applied by Breaker. Assuming that Maker has a winning strategy, we further assume that his goal is to win the game as soon as possible and that Breaker's goal is the opposite. We say that the players play \emph{optimally}, if they play complying with their goals. The \emph{winning number of Maker}, denoted by $\w_M^M(\cH)$ (resp., $\w_M^B(\cH)$), is the minimum number of his moves he needs to win the game if both players play optimally and Maker (resp., Breaker) starts the game~\cite{bujtas_fast_2022}. If Maker has no winning strategy as a first (resp., second) player, we set $\w_M^M(\cH)=\infty$ (resp., $\w_M^B(\cH)=\infty$). These general winning numbers were introduced in~\cite{bujtas_fast_2022}, but the problem of how fast Maker can win was studied earlier in several papers including~\cite{beck_positional_1981, clemens_fast_2012, clemens_how_2018, hefetz_fast_2009}. Note that $\gsmb(G)= \w_M^B(\cH_G)$ and $\gsmb'(G) = \w_M^M(\cH_G)$.

\subsection{Known results}

In this subsection we collect known results that will be used later. 

\begin{proposition}[\cite{bujtas_fast_2022}] \label{prop:delete-shrink}
	Suppose that a Maker-Breaker game is played on a hypergraph $\cH$ and $v \in V(\cH)$.
	\begin{itemize}
		\item[$(i)$] If Breaker is the first player and she plays $v$, the continuation of the game corresponds to a Maker-start game on $\cH-v$. In particular, if $v$ is an optimal first move of Breaker, then $\w_M^B(\cH)= \w_M^M(\cH -v)$.
		\item[$(ii)$] If Maker is the first player and he plays $v$, the continuation of the game corresponds to a Breaker-start game on $\cH \mid v$. In particular, if $v$ is an optimal first move of Maker, then $\w_M^M(\cH)= \w_M^B(\cH \mid v)+1$. Further, Maker wins the game with this move $v$ if and only if $\cH \mid v$ contains an empty hyperedge but $\cH$ does not.
	\end{itemize}
\end{proposition}

\begin{proposition}[\cite{bujtas_fast_2022}]
	\label{prop:delete}
	Let $\cH_1$ and $\cH_2$ be two hypergraphs on the same vertex set. 
	\begin{itemize}
		\item[$(i)$] If $E(\cH_1) \subseteq E(\cH_2)$, then $\w_M^M(\cH_1)\ge  \w_M^M(\cH_2)$ and $\w_M^B(\cH_1)\ge  \w_M^B(\cH_2)$ holds.
		\item[$(ii)$] Suppose that for each $e \in E(\cH_1)$ there exists an edge $e'\in E(\cH_2)$ such that $e'\subseteq e$. Then, $\w_M^M(\cH_1)\ge  \w_M^M(\cH_2)$ and $\w_M^B(\cH_1)\ge  \w_M^B(\cH_2)$ holds.
	\end{itemize}
\end{proposition}

\begin{proposition}[\cite{bujtas_fast_2022}]
	\label{prop:H-comp}
	Let $\cH$ be a disconnected hypergraph that consists of components $\cH_1, \dots, \cH_\ell$ such that  $\w_M^M(\cH_1) \le \dots \le \w_M^M(\cH_\ell)$. Then, the following holds:
	$$\w_M^M(\cH) =\w_M^M(\cH_1) \qquad  {\rm and} \qquad \w_M^M(\cH_1) \le \w_M^B(\cH) \le \w_M^M(\cH_2).$$
\end{proposition}		
We remark that Proposition~\ref{prop:H-comp} and its proof in~\cite{bujtas_fast_2022} directly imply that one of the optimal strategies of Maker is to play (optimally) on the component $\cH_1$. 

\begin{proposition}[\cite{bujtas_fast_2022}]
	\label{prop:comp} 
	If a disconnected graph $G$ consists of components $G_1, \dots, G_\ell$ and $\gsmb'(G_1) \le \dots  \le \gsmb'(G_\ell)$, then the following statements hold:
	\begin{itemize}
		\item[$(i)$] $\gsmb'(G)=\gsmb'(G_1)$;
		\item[$(ii)$] $\gsmb'(G_1) \le \gsmb(G)\le \gsmb'(G_2)$.	
	\end{itemize}	
\end{proposition}

\begin{proposition}[\cite{bujtas_fast_2022}] \label{prop:cut}
	\begin{itemize}
		\item[$(i)$] Let $G'$ be a graph obtained from $G$ by removing a weak support vertex and the adjacent leaf. Then, the following inequalities hold:
		$$\gsmb'(G)-1 \le \gsmb'(G') \le \gsmb'(G) \qquad {\rm and} \qquad \gsmb(G') \le \gsmb(G).$$
		\item[$(ii)$] Let $G'$ be a graph obtained from $G$ by removing a weak support vertex of degree $2$ and the adjacent leaf. 
		Then, $\gsmb(G)-1 \le \gsmb(G')$ holds.
		\item[$(iii)$] Let $v$ be a cut vertex in a connected graph $G$. If $G_1, \dots, G_\ell$ are the components of $G-v$ indexed so that $\gsmb'(G_1)\le  \dots \le \gsmb'(G_\ell)$, then $$\gsmb'(G) \le \gsmb'(G_2)+1.$$
	\end{itemize}
\end{proposition}

The proof of~\cite[Theorem 22]{duchene_makerbreaker_2020} yields the following statement that can also be deduced from \cite[Theorem 4.5]{gledel_makerbreaker_2019}.

\begin{theorem} [\cite{duchene_makerbreaker_2020, gledel_makerbreaker_2019}]
	\label{thm:outcometrees} 
	For a tree $T$, 
	\begin{itemize} 
		\item[$(i)$] $\gsmb(T) <\infty$ if and only if $R(T)$ contains at least two strong support vertices;
		\item[$(ii)$] $\gsmb'(T) <\infty$ if and only if $T$ does not admit a perfect matching.
	\end{itemize} 
\end{theorem}

\noindent
Note that, by Theorem~\ref{thm:outcometrees}(ii), $\gsmb'(P_n)=\infty$ if $n$ is even.

Applying the ``pairing strategy'' from~\cite{hefetz_positional_2014}, Dominator has a winning strategy in both the D-game and S-game if the graph has a perfect matching. This fact can be extended as follows. 

\begin{lemma}[\cite{bujtas_fast_2022}]
	\label{lem:pairing}
	Consider an MBD game on $G$ and let $X$ and $Y$ be the sets of vertices played by Dominator and Staller, respectively, until a moment during the game. If there exists a matching $M$ in $G-(X\cup Y)$ such that $V(G) \setminus V(M) \subseteq N_G[X]$,
	then Dominator has a strategy to win the continuation of the game, no matter who plays the next vertex.
\end{lemma} 

\begin{theorem}[\cite{bujtas_fast_2022}] \label{thm:path}
	If $n$ is an odd positive integer, then
	\begin{align*} \label{form:path}
		\gsmb'(P_n)=\left\lfloor \log_2n \right\rfloor +1.
	\end{align*}
	Moreover, Staller has an optimal strategy in the S-game on $P_n$ such that she wins on the closed neighborhood of a vertex $v$ that is at an even distance from the ends of the path.
\end{theorem}

If $n \ge 3$ and $n$ is odd, we may use the formula   $\lceil \log_2 n \rceil$ instead of $\left\lfloor \log_2n \right\rfloor +1$. In later sections we also frequently write $\lceil \log_2 n \rceil$ instead of $\left\lfloor \log_2(n-1) \right\rfloor +1$ if $n\ge 2$ is even. 

\section{Characterizing trees with $\gsmb'=k$}
\label{sec:trees}

In \cite{gledel_makerbreaker_2019}, exact formulas for $\gmb(T)$ and $\gmb'(T)$ were proved if $T$ is a tree. These values and the corresponding optimal strategies of Dominator are not difficult to determine by iteratively removing a pendant path $P_2$ and constructing the residual graph $R(T)$. On the other hand, as Theorem~\ref{thm:path} and its proof given in~\cite{bujtas_fast_2022} indicate, determining $\gsmb(T)$ and $\gsmb'(T)$ for an arbitrary tree $T$ turns out to be a much more difficult problem. 

Thus, the aim of this section is, for a given integer $k$, to characterize the trees $T$ with $\gsmb'(T)=k$. Combining this result with properties of closed neighborhood hypergraphs, we also establish a sufficient condition under which $\gsmb'(G)\le k$ holds for an arbitrary graph $G$.  

First, we state a lemma that will be used in the proof of the main theorem.
\begin{lemma}\label{lem:game-on-tree}
	Let $T$ be a tree with $2 \le \gsmb'(T) <\infty$ and suppose that $s_1$ is an optimal first move of Staller in the S-game on $T$.
	\begin{itemize}
		\item[$(i)$] If Dominator's response is playing a vertex $d_1$ which is a neighbor of $s_1$, then Staller's optimal next move $s_2$ and $d_1$ belong to different components of $T-s_1$. 
		\item[$(ii)$] After Staller's optimal move $s_1$, there exists an optimal response $d_1'$ of Dominator such that $d_1' \in N_T(s_1)$.
	\end{itemize}
\end{lemma}
\begin{proof}
 Let $T_1, \dots ,T_\ell$ be the components of $T-s_1$.\\
$(i)$  Suppose for a contradiction that $d_1$ and $s_2$ belong to the same component, say $T_1$. Proposition~\ref{prop:delete-shrink} implies that, after the moves $s_1$ and $d_1$, the game continues as a Maker-start game with the winning sets in $\cH'=(\cH_T \! \mid \! s_1) -d_1$. As $s_1$ is an optimal start vertex for Staller, $\gsmb'(T)\ge  1+\w_M^M(\cH')$. Further, since by our assumption, $s_1$ and $d_1$ are adjacent vertices in $T$, the hypergraph $\cH'$ contains the components   $\cH_{T_2}, \dots , \cH_{T_\ell}$ and  all components of $\cH_1=\cH_{T_1}-d_1$. By Proposition~\ref{prop:H-comp},  $\w_M^M(\cH')$ equals the minimum of the values $\w_M^M(\cH_1), \w_M^M(\cH_{T_2}), \dots,  \w_M^M(\cH_{T_\ell})$, and one the optimal strategies of Staller (Maker) is to play in the same (appropriate) component of $\cH'$ starting with her move $s_2$. As $s_2 \in V(\cH_1)$ is supposed to be an optimal move, $\w_M^M(\cH')=\w_M^M(\cH_1)$. On the other hand, $\cH_1$ is a subhypergraph of $\cH_{T}$ and Proposition~\ref{prop:delete} thus implies $\w_M^M(\cH_1)\ge \w_M^M(\cH_{T})$. This gives the following contradiction:
$$\gsmb'(T)= \w_M^M(\cH_{T}) \ge 1+\w_M^M(\cH')=1+\w_M^M(\cH_1) \ge 1+ \gsmb'(T).
$$
$(ii)$ Suppose $d_1$ is an optimal response of Dominator to the move $s_1$ and  $d_1 \notin N_T(s_1)$. We may also suppose, without loss of generality, that $d_1 \in V(T_1)$. Let $d_1'$ be the neighbor of $s_1$ from the subtree $T_1$.
We will show, by using imagination strategy, that $d'_1$ is also an optimal response of Dominator. 
Let Game 1 be an S-game on $T$ and let Staller start by playing $s_1$ and Dominator's response is $d_1'$. After that, Staller plays optimally as follows: 
\begin{itemize}
	\item[(*)] Staller selects a component of $(\cH_T|s_1) - d_1'$ with the possible smallest winning number and plays optimally all of her moves in this component.  
\end{itemize}
Note that by Proposition~\ref{prop:H-comp} the strategy (*) is indeed an optimal one for Staller from the second move. 

Let Game 2 be an S-game on $T$ in which Dominator plays optimally and his response to the move $s_1$ is $d_1$. In the continuation, Staller always selects an optimal move $s_i$ in Game 1 and copies her move to Game 2 if possible, while Dominator selects an optimal move $d_i$ in Game 2 and copies it to Game 1 if possible. 
Let $t_k$ be the number of moves Staller needs to win in Game $k$ for $k \in [2]$.

By $(i)$, the optimal second move $s_2$ of Staller in Game 1 cannot belong to $T_1$. Let us assume $s_2 \in V(T_j)$ with $j \neq 1$. Then according to (*), all the remaining moves of Staller will be in $T_j$. In particular, Staller never plays the vertex $d_1$. On the other hand, if Dominator plays $d_i=d_1'$ in Game 2, then we replace this move by $d_1$ in Game 1. Otherwise Dominator's move can always be copied into Game 1. 

At the end, Staller wins Game 1 with her $(t_1)^{\rm st}$ move by claiming a winning set from $\cH_{T_j}$.
As this winning set is also claimed in Game 2 (or Staller has already won in Game 2), Staller wins Game 2 in $t_2 \le t_1$ moves.
Since Staller plays optimally throughout Game 1, we infer  $\gsmb'(T) \ge t_1$. 
Similarly, $t_2 \ge \gsmb'(T)$ holds because Dominator plays optimally in Game 2.
Therefore, $\gsmb'(T) \ge t_1 \ge t_2 \ge \gsmb'(T)$ that implies $\gsmb'(T)=t_1$.  
We may conclude that $d'_1 \in N_T(s_1)$ is also an optimal move for Dominator after $s_1$.
\end{proof}

\begin{definition} \label{def:S_k}
	We define a \emph{family $\cS_k$} of graphs for each positive integer $k$, and for each $S\in \cS_k$ a subset of vertices $X(S)$, as follows.
	\begin{itemize}
		\item Let $\cS_1=\{P_1\}$ and let $X(P_1)$ contain the only vertex of $P_1$.
		\item For $k \ge 2$, let $\cS_k^*$ be the set of graphs that can be obtained in the following way. Take two vertex disjoint graphs $S^1 \in \cS_{k-1}$,  $S^2 \in \cup_{i=1}^{k-1}\cS_i$, and a new vertex $z_k$. Choose a vertex $x^j$ from $X(S^j)$, for each $j \in [2]$. Then, a graph $S$ from $\cS_k^*$ is obtained by making adjacent the \emph{origin vertex} $z_k$ with $x^1$ and $x^2$. We define $X(S)=X(S^1) \cup X(S^2)$.
		\item  The family  $\cS_k$ is obtained as $\cS_k^* \setminus \cup_{i=1}^{k-1}\cS_i$.
	\end{itemize}
	We also introduce the notation $\cS= \bigcup_{k=1}^\infty \cS_k $  and say that $S \in \cS$ is a \emph{substructure in the graph $G$} if $S$ is a subgraph of $G$ and $\deg_G(x)=\deg_S(x)$ holds for each $x \in X(S)$.  We further say that the vertices from $X(S)$ have \emph{fixed degrees}.
\end{definition}

In Fig.~\ref{fig:Sfamily}, families $\cS_1$, $\cS_2$, $\cS_3$, and $\cS_4$, are presented, where the black vertices are the vertices with fixed degrees. Note that the vertices in $V(S)\setminus X(S)$ may be incident with any number of additional edges in $G$. Due to the recursive definition of $\cS_k$, each of the following statements can be easily verified by proceeding with induction on $k$, cf.~Fig.~\ref{fig:Sfamily} again. 

\begin{center}
	\begin{figure}[ht!]
		\definecolor{ffffff}{rgb}{1.,1.,1.}
		\begin{tikzpicture}[scale=0.78,line cap=round,line join=round,>=triangle 45,x=0.8cm,y=0.5cm]
			\draw [line width=1.2pt] (7.,2.)-- (8.,2.);
			\draw [line width=1.2pt] (8.,2.)-- (9.,2.);
			\draw [line width=1.2pt] (-2.,0.)-- (-1.,0.);
			\draw [line width=1.2pt] (-1.,0.)-- (0.,0.);
			\draw [line width=1.2pt] (0.,0.)-- (1.,0.);
			\draw [line width=1.2pt] (1.,0.)-- (2.,0.);
			\draw [line width=1.2pt] (5.,0.)-- (6.,0.);
			\draw [line width=1.2pt] (6.,0.)-- (7.,0.);
			\draw [line width=1.2pt] (7.,0.)-- (8.,0.);
			\draw [line width=1.2pt] (8.,0.)-- (9.,0.);
			\draw [line width=1.2pt] (9.,0.)-- (10.,0.);
			\draw [line width=1.2pt] (10.,0.)-- (11.,0.);
			\draw [line width=1.2pt] (-2.,-2.)-- (-1.,-2.);
			\draw [line width=1.2pt] (-1.,-2.)-- (0.,-2.);
			\draw [line width=1.2pt] (0.,-2.)-- (1.,-2.);
			\draw [line width=1.2pt] (1.,-2.)-- (2.,-2.);
			\draw [line width=1.2pt] (3.,-3.)-- (4.,-2.);
			\draw [line width=1.2pt] (3.,-3.)-- (0.,-2.);
			\draw [line width=1.2pt] (-2.,-5.)-- (-1.,-5.);
			\draw [line width=1.2pt] (-1.,-5.)-- (0.,-5.);
			\draw [line width=1.2pt] (0.,-5.)-- (1.,-5.);
			\draw [line width=1.2pt] (1.,-5.)-- (2.,-5.);
			\draw [line width=1.2pt] (4.,-5.)-- (5.,-5.);
			\draw [line width=1.2pt] (5.,-5.)-- (6.,-5.);
			\draw [line width=1.2pt] (3.,-5.)-- (4.,-5.);
			\draw [line width=1.2pt] (2.,-5.)-- (3.,-5.);
			\draw [line width=1.2pt] (-2.,-7.)-- (-1.,-7.);
			\draw [line width=1.2pt] (-1.,-7.)-- (0.,-7.);
			\draw [line width=1.2pt] (0.,-7.)-- (1.,-7.);
			\draw [line width=1.2pt] (1.,-7.)-- (2.,-7.);
			\draw [line width=1.2pt] (4.,-7.)-- (5.,-7.);
			\draw [line width=1.2pt] (5.,-7.)-- (6.,-7.);
			\draw [line width=1.2pt] (3.,-8.)-- (0.,-7.);
			\draw [line width=1.2pt] (3.,-8.)-- (4.,-7.);
			\draw [line width=1.2pt] (-2.,-10.)-- (-1.,-10.);
			\draw [line width=1.2pt] (-1.,-10.)-- (0.,-10.);
			\draw [line width=1.2pt] (0.,-10.)-- (1.,-10.);
			\draw [line width=1.2pt] (1.,-10.)-- (2.,-10.);
			\draw [line width=1.2pt] (4.,-10.)-- (5.,-10.);
			\draw [line width=1.2pt] (5.,-10.)-- (6.,-10.);
			\draw [line width=1.2pt] (7.,-10.)-- (6.,-10.);
			\draw [line width=1.2pt] (7.,-10.)-- (8.,-10.);
			\draw [line width=1.2pt] (3.,-10.)-- (4.,-10.);
			\draw [line width=1.2pt] (3.,-10.)-- (2.,-10.);
			\draw [line width=1.2pt] (-2.,-12.)-- (-1.,-12.);
			\draw [line width=1.2pt] (-1.,-12.)-- (0.,-12.);
			\draw [line width=1.2pt] (0.,-12.)-- (1.,-12.);
			\draw [line width=1.2pt] (1.,-12.)-- (2.,-12.);
			\draw [line width=1.2pt] (4.,-12.)-- (5.,-12.);
			\draw [line width=1.2pt] (5.,-12.)-- (6.,-12.);
			\draw [line width=1.2pt] (6.,-12.)-- (7.,-12.);
			\draw [line width=1.2pt] (7.,-12.)-- (8.,-12.);
			\draw [line width=1.2pt] (0.,-12.)-- (3.,-13.);
			\draw [line width=1.2pt] (3.,-13.)-- (4.,-12.);
			\draw [line width=1.2pt] (-2.,-15.)-- (-1.,-15.);
			\draw [line width=1.2pt] (-1.,-15.)-- (0.,-15.);
			\draw [line width=1.2pt] (0.,-15.)-- (1.,-15.);
			\draw [line width=1.2pt] (1.,-15.)-- (2.,-15.);
			\draw [line width=1.2pt] (4.,-15.)-- (5.,-15.);
			\draw [line width=1.2pt] (5.,-15.)-- (6.,-15.);
			\draw [line width=1.2pt] (6.,-15.)-- (7.,-15.);
			\draw [line width=1.2pt] (7.,-15.)-- (8.,-15.);
			\draw [line width=1.2pt] (3.,-16.)-- (6.,-15.);
			\draw [line width=1.2pt] (3.,-16.)-- (0.,-15.);
			
			\draw [line width=1.2pt] (-2.,-17.)-- (-1.,-17.);
			\draw [line width=1.2pt] (-1.,-17.)-- (0.,-17.);
			\draw [line width=1.2pt] (0.,-17.)-- (1.,-17.);
			\draw [line width=1.2pt] (1.,-17.)-- (2.,-17.);
			\draw [line width=1.2pt] (2.,-17.)-- (3.,-17.);
			\draw [line width=1.2pt] (3.,-17.)-- (4.,-17.);
			\draw [line width=1.2pt] (6.,-17.)-- (7.,-17.);
			\draw [line width=1.2pt] (7.,-17.)-- (8.,-17.);
			\draw [line width=1.2pt] (5.,-18.)-- (6.,-17.);
			\draw [line width=1.2pt] (5.,-18.)-- (2.,-17.);
			
			\draw [line width=1.2pt] (-2.,-19.)-- (-1.,-19.);
			\draw [line width=1.2pt] (-1.,-19.)-- (0.,-19.);
			\draw [line width=1.2pt] (0.,-19.)-- (1.,-19.);
			\draw [line width=1.2pt] (1.,-19.)-- (2.,-19.);
			\draw [line width=1.2pt] (2.,-19.)-- (3.,-19.);
			\draw [line width=1.2pt] (3.,-19.)-- (4.,-19.);
			\draw [line width=1.2pt] (6.,-19.)-- (7.,-19.);
			\draw [line width=1.2pt] (7.,-19.)-- (8.,-19.);
			\draw [line width=1.2pt] (8.,-19.)-- (9.,-19.);
			\draw [line width=1.2pt] (9.,-19.)-- (10.,-19.);
			\draw [line width=1.2pt] (5.,-19.)-- (6.,-19.);
			\draw [line width=1.2pt] (5.,-19.)-- (4.,-19.);
			\draw [line width=1.2pt] (-2.,-21.)-- (-1.,-21.);
			\draw [line width=1.2pt] (-1.,-21.)-- (0.,-21.);
			\draw [line width=1.2pt] (0.,-21.)-- (1.,-21.);
			\draw [line width=1.2pt] (1.,-21.)-- (2.,-21.);
			\draw [line width=1.2pt] (2.,-21.)-- (3.,-21.);
			\draw [line width=1.2pt] (3.,-21.)-- (4.,-21.);
			\draw [line width=1.2pt] (6.,-21.)-- (7.,-21.);
			\draw [line width=1.2pt] (7.,-21.)-- (8.,-21.);
			\draw [line width=1.2pt] (8.,-21.)-- (9.,-21.);
			\draw [line width=1.2pt] (9.,-21.)-- (10.,-21.);
			
			\draw [line width=1.2pt] (-2.,-24.)-- (-1.,-24.);
			\draw [line width=1.2pt] (-1.,-24.)-- (0.,-24.);
			\draw [line width=1.2pt] (0.,-24.)-- (1.,-24.);
			\draw [line width=1.2pt] (1.,-24.)-- (2.,-24.);
			\draw [line width=1.2pt] (2.,-24.)-- (3.,-24.);
			\draw [line width=1.2pt] (3.,-24.)-- (4.,-24.);
			\draw [line width=1.2pt] (6.,-24.)-- (7.,-24.);
			\draw [line width=1.2pt] (7.,-24.)-- (8.,-24.);
			\draw [line width=1.2pt] (8.,-24.)-- (9.,-24.);
			\draw [line width=1.2pt] (9.,-24.)-- (10.,-24.);
			\draw [line width=1.2pt] (5.,-22.)-- (8.,-21.);
			\draw [line width=1.2pt] (6.,-24.)-- (5.,-25.);
			\draw [line width=1.2pt] (-2.,-27.)-- (-1.,-27.);
			\draw [line width=1.2pt] (-1.,-27.)-- (0.,-27.);
			\draw [line width=1.2pt] (0.,-27.)-- (1.,-27.);
			\draw [line width=1.2pt] (1.,-27.)-- (2.,-27.);
			\draw [line width=1.2pt] (2.,-27.)-- (3.,-27.);
			\draw [line width=1.2pt] (3.,-27.)-- (4.,-27.);
			\draw [line width=1.2pt] (6.,-27.)-- (7.,-27.);
			\draw [line width=1.2pt] (7.,-27.)-- (8.,-27.);
			\draw [line width=1.2pt] (8.,-27.)-- (9.,-27.);
			\draw [line width=1.2pt] (9.,-27.)-- (10.,-27.);
			\draw [line width=1.2pt] (5.,-28.)-- (8.,-27.);
			\draw [line width=1.2pt] (-2.,-30.)-- (-1.,-30.);
			\draw [line width=1.2pt] (-1.,-30.)-- (0.,-30.);
			\draw [line width=1.2pt] (0.,-30.)-- (1.,-30.);
			\draw [line width=1.2pt] (1.,-30.)-- (2.,-30.);
			\draw [line width=1.2pt] (2.,-30.)-- (3.,-30.);
			\draw [line width=1.2pt] (3.,-30.)-- (4.,-30.);
			\draw [line width=1.2pt] (4.,-30.)-- (5.,-30.);
			\draw [line width=1.2pt] (5.,-30.)-- (6.,-30.);
			\draw [line width=1.2pt] (6.,-30.)-- (7.,-30.);
			\draw [line width=1.2pt] (7.,-30.)-- (8.,-30.);
			\draw [line width=1.2pt] (8.,-30.)-- (9.,-30.);
			\draw [line width=1.2pt] (9.,-30.)-- (10.,-30.);
			\draw [line width=1.2pt] (10.,-30.)-- (11.,-30.);
			\draw [line width=1.2pt] (11.,-30.)-- (12.,-30.);
			\draw [line width=1.2pt] (-2.,-32.)-- (-1.,-32.);
			\draw [line width=1.2pt] (-1.,-32.)-- (0.,-32.);
			\draw [line width=1.2pt] (0.,-32.)-- (1.,-32.);
			\draw [line width=1.2pt] (1.,-32.)-- (2.,-32.);
			\draw [line width=1.2pt] (2.,-32.)-- (3.,-32.);
			\draw [line width=1.2pt] (3.,-32.)-- (4.,-32.);
			\draw [line width=1.2pt] (6.,-32.)-- (7.,-32.);
			\draw [line width=1.2pt] (7.,-32.)-- (8.,-32.);
			\draw [line width=1.2pt] (8.,-32.)-- (9.,-32.);
			\draw [line width=1.2pt] (9.,-32.)-- (10.,-32.);
			\draw [line width=1.2pt] (10.,-32.)-- (11.,-32.);
			\draw [line width=1.2pt] (11.,-32.)-- (12.,-32.);
			\draw [line width=1.2pt] (8.,-32.)-- (5.,-33.);
			\draw [line width=1.2pt] (-2.,-35.)-- (-1.,-35.);
			\draw [line width=1.2pt] (-1.,-35.)-- (0.,-35.);
			\draw [line width=1.2pt] (0.,-35.)-- (1.,-35.);
			\draw [line width=1.2pt] (1.,-35.)-- (2.,-35.);
			\draw [line width=1.2pt] (2.,-35.)-- (3.,-35.);
			\draw [line width=1.2pt] (3.,-35.)-- (4.,-35.);
			\draw [line width=1.2pt] (6.,-35.)-- (7.,-35.);
			\draw [line width=1.2pt] (7.,-35.)-- (8.,-35.);
			\draw [line width=1.2pt] (8.,-35.)-- (9.,-35.);
			\draw [line width=1.2pt] (9.,-35.)-- (10.,-35.);
			\draw [line width=1.2pt] (10.,-35.)-- (11.,-35.);
			\draw [line width=1.2pt] (11.,-35.)-- (12.,-35.);
			\draw [line width=1.2pt] (5.,-36.)-- (8.,-35.);
			\draw [line width=1.2pt] (5.,-22.)-- (4.,-21.);
			\draw [line width=1.2pt] (5.,-36.)-- (2.,-35.);
			\draw [line width=1.2pt] (5.,-33.)-- (4.,-32.);
			\draw [line width=1.2pt] (5.,-28.)-- (2.,-27.);
			\draw [line width=1.2pt] (5.,-25.)-- (2.,-24.);
			\draw (-4,2.8) node[anchor=north west] {$\mathcal{S}_1:$};
			\draw (4.5,2.8) node[anchor=north west] {$\mathcal{S}_2:$};
			\draw (-4,1) node[anchor=north west] {$\mathcal{S}_3:$};
			\draw (-4,-1) node[anchor=north west] {$\mathcal{S}_4:$};
			\draw (0.5,2.5) node[anchor=north west] {$P_1$};
			\draw (10.5,2.5) node[anchor=north west] {$P_3$};
			\draw (2.5,0.6) node[anchor=north west] {$P_5$};
			\draw (11.5,0.6) node[anchor=north west] {$P_7$};
			\draw (6.5,-4.3) node[anchor=north west] {$P_9$};
			\draw (8.5,-9.2) node[anchor=north west] {$P_{11}$};
			\draw (10.8,-17.2) node[anchor=north west] {$P_{13}$};
			\draw (12.4,-28.4) node[anchor=north west] {$P_{15}$};
			\begin{scriptsize}
				\draw [fill=black] (0.,2.) circle (2.5pt);
				\draw [fill=black] (7.,2.) circle (2.5pt);
				\draw [fill=ffffff] (8.,2.) circle (2.5pt);
				\draw [fill=black] (9.,2.) circle (2.5pt);
				\draw [fill=black] (-2.,0.) circle (2.5pt);
				\draw [fill=ffffff] (-1.,0.) circle (2.5pt);
				\draw [fill=black] (0.,0.) circle (2.5pt);
				\draw [fill=ffffff] (1.,0.) circle (2.5pt);
				\draw [fill=black] (2.,0.) circle (2.5pt);
				\draw [fill=black] (5.,0.) circle (2.5pt);
				\draw [fill=ffffff] (6.,0.) circle (2.5pt);
				\draw [fill=black] (7.,0.) circle (2.5pt);
				\draw [fill=ffffff] (8.,0.) circle (2.5pt);
				\draw [fill=black] (9.,0.) circle (2.5pt);
				\draw [fill=ffffff] (10.,0.) circle (2.5pt);
				\draw [fill=black] (11.,0.) circle (2.5pt);
				\draw [fill=black] (-2.,-2.) circle (2.5pt);
				\draw [fill=ffffff] (-1.,-2.) circle (2.5pt);
				\draw [fill=black] (0.,-2.) circle (2.5pt);
				\draw [fill=ffffff] (1.,-2.) circle (2.5pt);
				\draw [fill=black] (2.,-2.) circle (2.5pt);
				\draw [fill=black] (4.,-2.) circle (2.5pt);
				\draw [fill=ffffff] (3.,-3.) circle (2.5pt);
				\draw [fill=black] (-2.,-5.) circle (2.5pt);
				\draw [fill=ffffff] (-1.,-5.) circle (2.5pt);
				\draw [fill=black] (0.,-5.) circle (2.5pt);
				\draw [fill=ffffff] (1.,-5.) circle (2.5pt);
				\draw [fill=black] (2.,-5.) circle (2.5pt);
				\draw [fill=black] (4.,-5.) circle (2.5pt);
				\draw [fill=ffffff] (5.,-5.) circle (2.5pt);
				\draw [fill=black] (6.,-5.) circle (2.5pt);
				\draw [fill=ffffff] (3.,-5.) circle (2.5pt);
				\draw [fill=black] (-2.,-7.) circle (2.5pt);
				\draw [fill=ffffff] (-1.,-7.) circle (2.5pt);
				\draw [fill=black] (0.,-7.) circle (2.5pt);
				\draw [fill=ffffff] (1.,-7.) circle (2.5pt);
				\draw [fill=black] (2.,-7.) circle (2.5pt);
				\draw [fill=black] (4.,-7.) circle (2.5pt);
				\draw [fill=ffffff] (5.,-7.) circle (2.5pt);
				\draw [fill=black] (6.,-7.) circle (2.5pt);
				\draw [fill=ffffff] (3.,-8.) circle (2.5pt);
				\draw [fill=black] (-2.,-10.) circle (2.5pt);
				\draw [fill=ffffff] (-1.,-10.) circle (2.5pt);
				\draw [fill=black] (0.,-10.) circle (2.5pt);
				\draw [fill=ffffff] (1.,-10.) circle (2.5pt);
				\draw [fill=black] (2.,-10.) circle (2.5pt);
				\draw [fill=black] (4.,-10.) circle (2.5pt);
				\draw [fill=ffffff] (5.,-10.) circle (2.5pt);
				\draw [fill=black] (6.,-10.) circle (2.5pt);
				\draw [fill=ffffff] (7.,-10.) circle (2.5pt);
				\draw [fill=black] (8.,-10.) circle (2.5pt);
				\draw [fill=ffffff] (3.,-10.) circle (2.5pt);
				\draw [fill=black] (-2.,-12.) circle (2.5pt);
				\draw [fill=ffffff] (-1.,-12.) circle (2.5pt);
				\draw [fill=black] (0.,-12.) circle (2.5pt);
				\draw [fill=ffffff] (1.,-12.) circle (2.5pt);
				\draw [fill=black] (2.,-12.) circle (2.5pt);
				\draw [fill=black] (4.,-12.) circle (2.5pt);
				\draw [fill=ffffff] (5.,-12.) circle (2.5pt);
				\draw [fill=black] (6.,-12.) circle (2.5pt);
				\draw [fill=ffffff] (7.,-12.) circle (2.5pt);
				\draw [fill=black] (8.,-12.) circle (2.5pt);
				\draw [fill=ffffff] (3.,-13.) circle (2.5pt);
				\draw [fill=black] (-2.,-15.) circle (2.5pt);
				\draw [fill=ffffff] (-1.,-15.) circle (2.5pt);
				\draw [fill=black] (0.,-15.) circle (2.5pt);
				\draw [fill=ffffff] (1.,-15.) circle (2.5pt);
				\draw [fill=black] (2.,-15.) circle (2.5pt);
				\draw [fill=black] (4.,-15.) circle (2.5pt);
				\draw [fill=ffffff] (5.,-15.) circle (2.5pt);
				\draw [fill=black] (6.,-15.) circle (2.5pt);
				\draw [fill=ffffff] (7.,-15.) circle (2.5pt);
				\draw [fill=black] (8.,-15.) circle (2.5pt);
				\draw [fill=ffffff] (3.,-16.) circle (2.5pt);
				
				\draw [fill=black] (-2.,-17.) circle (2.5pt);
				\draw [fill=ffffff] (-1.,-17.) circle (2.5pt);
				\draw [fill=black] (0.,-17.) circle (2.5pt);
				\draw [fill=ffffff] (1.,-17.) circle (2.5pt);
				\draw [fill=black] (2.,-17.) circle (2.5pt);
				\draw [fill=ffffff] (3.,-17.) circle (2.5pt);
				\draw [fill=black] (4.,-17.) circle (2.5pt);
				\draw [fill=black] (6.,-17.) circle (2.5pt);
				\draw [fill=ffffff] (7.,-17.) circle (2.5pt);
				\draw [fill=black] (8.,-17.) circle (2.5pt);
				\draw [fill=ffffff] (5.,-18.) circle (2.5pt);
				
				\draw [fill=black] (-2.,-19.) circle (2.5pt);
				\draw [fill=ffffff] (-1.,-19.) circle (2.5pt);
				\draw [fill=black] (0.,-19.) circle (2.5pt);
				\draw [fill=ffffff] (1.,-19.) circle (2.5pt);
				\draw [fill=black] (2.,-19.) circle (2.5pt);
				\draw [fill=ffffff] (3.,-19.) circle (2.5pt);
				\draw [fill=black] (4.,-19.) circle (2.5pt);
				\draw [fill=black] (6.,-19.) circle (2.5pt);
				\draw [fill=ffffff] (7.,-19.) circle (2.5pt);
				\draw [fill=black] (8.,-19.) circle (2.5pt);
				\draw [fill=ffffff] (9.,-19.) circle (2.5pt);
				\draw [fill=black] (10.,-19.) circle (2.5pt);
				\draw [fill=ffffff] (5.,-19.) circle (2.5pt);
				
				\draw [fill=black] (-2.,-21.) circle (2.5pt);
				\draw [fill=ffffff] (-1.,-21.) circle (2.5pt);
				\draw [fill=black] (0.,-21.) circle (2.5pt);
				\draw [fill=ffffff] (1.,-21.) circle (2.5pt);
				\draw [fill=black] (2.,-21.) circle (2.5pt);
				\draw [fill=ffffff] (3.,-21.) circle (2.5pt);
				\draw [fill=black] (4.,-21.) circle (2.5pt);
				\draw [fill=black] (6.,-21.) circle (2.5pt);
				\draw [fill=ffffff] (7.,-21.) circle (2.5pt);
				\draw [fill=black] (8.,-21.) circle (2.5pt);
				\draw [fill=ffffff] (9.,-21.) circle (2.5pt);
				\draw [fill=black] (10.,-21.) circle (2.5pt);
				\draw [fill=ffffff] (5.,-22.) circle (2.5pt);
				
				\draw [fill=black] (-2.,-24.) circle (2.5pt);
				\draw [fill=ffffff] (-1.,-24.) circle (2.5pt);
				\draw [fill=black] (0.,-24.) circle (2.5pt);
				\draw [fill=ffffff] (1.,-24.) circle (2.5pt);
				\draw [fill=black] (2.,-24.) circle (2.5pt);
				\draw [fill=ffffff] (3.,-24.) circle (2.5pt);
				\draw [fill=black] (4.,-24.) circle (2.5pt);
				\draw [fill=black] (6.,-24.) circle (2.5pt);
				\draw [fill=ffffff] (7.,-24.) circle (2.5pt);
				\draw [fill=black] (8.,-24.) circle (2.5pt);
				\draw [fill=ffffff] (9.,-24.) circle (2.5pt);
				\draw [fill=black] (10.,-24.) circle (2.5pt);
				\draw [fill=ffffff] (5.,-25.) circle (2.5pt);
				
				\draw [fill=black] (-2.,-27.) circle (2.5pt);
				\draw [fill=ffffff] (-1.,-27.) circle (2.5pt);
				\draw [fill=black] (0.,-27.) circle (2.5pt);
				\draw [fill=ffffff] (1.,-27.) circle (2.5pt);
				\draw [fill=black] (2.,-27.) circle (2.5pt);
				\draw [fill=ffffff] (3.,-27.) circle (2.5pt);
				\draw [fill=black] (4.,-27.) circle (2.5pt);
				\draw [fill=black] (6.,-27.) circle (2.5pt);
				\draw [fill=ffffff] (7.,-27.) circle (2.5pt);
				\draw [fill=black] (8.,-27.) circle (2.5pt);
				\draw [fill=ffffff] (9.,-27.) circle (2.5pt);
				\draw [fill=black] (10.,-27.) circle (2.5pt);
				\draw [fill=ffffff] (5.,-28.) circle (2.5pt);
				
				\draw [fill=black] (-2.,-30.) circle (2.5pt);
				\draw [fill=ffffff] (-1.,-30.) circle (2.5pt);
				\draw [fill=black] (0.,-30.) circle (2.5pt);
				\draw [fill=ffffff] (1.,-30.) circle (2.5pt);
				\draw [fill=black] (2.,-30.) circle (2.5pt);
				\draw [fill=ffffff] (3.,-30.) circle (2.5pt);
				\draw [fill=black] (4.,-30.) circle (2.5pt);
				\draw [fill=ffffff] (5.,-30.) circle (2.5pt);
				\draw [fill=black] (6.,-30.) circle (2.5pt);
				\draw [fill=ffffff] (7.,-30.) circle (2.5pt);
				\draw [fill=black] (8.,-30.) circle (2.5pt);
				\draw [fill=ffffff] (9.,-30.) circle (2.5pt);
				\draw [fill=black] (10.,-30.) circle (2.5pt);
				\draw [fill=ffffff] (11.,-30.) circle (2.5pt);
				\draw [fill=black] (12.,-30.) circle (2.5pt);
				
				\draw [fill=black] (-2.,-32.) circle (2.5pt);
				\draw [fill=ffffff] (-1.,-32.) circle (2.5pt);
				\draw [fill=black] (0.,-32.) circle (2.5pt);
				\draw [fill=ffffff] (1.,-32.) circle (2.5pt);
				\draw [fill=black] (2.,-32.) circle (2.5pt);
				\draw [fill=ffffff] (3.,-32.) circle (2.5pt);
				\draw [fill=black] (4.,-32.) circle (2.5pt);
				\draw [fill=ffffff] (5.,-33.) circle (2.5pt);
				\draw [fill=black] (6.,-32.) circle (2.5pt);
				\draw [fill=ffffff] (7.,-32.) circle (2.5pt);
				\draw [fill=black] (8.,-32.) circle (2.5pt);
				\draw [fill=ffffff] (9.,-32.) circle (2.5pt);
				\draw [fill=black] (10.,-32.) circle (2.5pt);
				\draw [fill=ffffff] (11.,-32.) circle (2.5pt);
				\draw [fill=black] (12.,-32.) circle (2.5pt);
				
				\draw [fill=black] (-2.,-35.) circle (2.5pt);
				\draw [fill=ffffff] (-1.,-35.) circle (2.5pt);
				\draw [fill=black] (0.,-35.) circle (2.5pt);
				\draw [fill=ffffff] (1.,-35.) circle (2.5pt);
				\draw [fill=black] (2.,-35.) circle (2.5pt);
				\draw [fill=ffffff] (3.,-35.) circle (2.5pt);
				\draw [fill=black] (4.,-35.) circle (2.5pt);
				\draw [fill=black] (6.,-35.) circle (2.5pt);
				\draw [fill=ffffff] (7.,-35.) circle (2.5pt);
				\draw [fill=black] (8.,-35.) circle (2.5pt);
				\draw [fill=ffffff] (9.,-35.) circle (2.5pt);
				\draw [fill=black] (10.,-35.) circle (2.5pt);
				\draw [fill=ffffff] (11.,-35.) circle (2.5pt);
				\draw [fill=black] (12.,-35.) circle (2.5pt);
				\draw [fill=ffffff] (5.,-36.) circle (2.5pt);
			\end{scriptsize}
		\end{tikzpicture}
		\caption{Families $\cS_1$, $\cS_2$, $\cS_3$, and $\cS_4$, where the black vertices are the vertices with fixed degrees.}
		\label{fig:Sfamily}
	\end{figure}	
\end{center}

\begin{observation} \label{obs:S_k}
	Let $S \in \cS_k$ for an integer $k \ge 1$. 
	\begin{itemize}
		\item[$(i)$] $S$ is a tree in which all leaves belong to the same partite class. Moreover, this partite class corresponds to $X(S)$.
		\item[$(ii)$] If $v \in V(S)\setminus X(S)$, then $\deg_S(v)=2$. In particular, $\deg_S(u)=2$ holds for every support vertex $u$ of $S$.
		\item[$(iii)$] If a tree $T$ contains a substructure $S$, then every leaf of $S$ is a leaf in $T$ and every support vertex of $S$ is a support vertex in $T$.
		\item[$(iv)$] $|V(S)|$ is odd and $|V(S)| \le 2^k -1$.
		\item[$(v)$]  Let $\cP=\{P_1, P_2, \dots \}$. Then $\cP \cap \cS_k= \{P_{2^{k-1}+1}, P_{2^{k-1}+3}, \dots , P_{2^{k}-1 }\}$.
	\end{itemize}
\end{observation}

Observation~\ref{obs:S_k}$(i)$ shows that, for each $k$, the family  $\cS_k$ can be considered simply as a family of graphs, without associating each $S\in \cS_k$ with a set $X(S) \subseteq V(S)$. The fixed vertex degrees are easy to identify by considering $\deg_S(v)$ for every vertex $v$ which is in the same partite class as the leaves of $S$. 

In the following proposition, we state a simple characterization for the family $\cS$. For it recall that the {\em subdivision graph}, $S(G)$, of a graph $G$ is the graph obtained from $G$ by subdividing each edge exactly once, cf.~\cite{burzio_subdivision_1998}.

\begin{proposition} \label{prop:cS}
	$\cS = \{S(G):\ G\ \mbox{is\ a\ tree}\}$.
\end{proposition}  

\begin{proof} 
Set $\cS' = \{S(G):\ G\ \mbox{is\ a\ tree}\}$. We are going to prove that $\cS = \cS'$. 

Clearly, the elements of $\cS'$ are trees. Moreover, if $T$ is an arbitrary tree with partite classes $A_1$ and $A_2$ such that $A_1$ contains at least one leaf or an isolated vertex, then it is straightforward to observe that $T \in \cS'$ if and only if\/ 
\begin{itemize}
	\item[$(\ast)$] $\deg_T(v)=2$ holds for every $v \in A_2$. 
\end{itemize}

If $T \in \cS$, Observation~\ref{obs:S_k}(i) and (ii) yield the desired property $(\ast)$ with $A_1=X(T)$ and $A_2=V(T)\setminus X(T)$. Hence $\cS \subseteq \cS'$. 

To prove that $\cS' \subseteq \cS$ also holds, suppose that $T\in \cS'$.  Then $T$ fulfils $(\ast)$. We proceed by induction to conclude $T \in \cS$. This clearly holds if the order of $T$ is at most $3$. Let $|V(T)|>3$. Then, $A_2 \neq \emptyset$, and we fix an arbitrarily chosen vertex $z$ from $A_2$. By our condition, $\deg_T(z)=2$ and hence $T-z$ consists of two components, say $T_1$ and $T_2$. Observe that $\deg_{T_1}(v)= \deg_T(v)$ holds for every $v \in V(T_1)$ except for the neighbor of $z$ from $T_1$.  Since $z \in A_2$, this neighbor is from $A_1$, and therefore, we infer that  $\deg_{T_1}(v)= \deg_T(v)=2$ remains true for all vertices $v$ contained in the partite class $V(T_1)\cap A_2$ of $T_1$. Applying the induction hypothesis to $T_1$, we conclude $T_1 \in \cS$. In the same way, we may obtain $T_2 \in \cS$. Definition~\ref{def:S_k} then shows that $T$ can be constructed from $T_1$, $T_2$, and the origin vertex $z$, and therefore $T\in \cS_k$ for an appropriate integer $k$. \end{proof}

\begin{theorem} \label{thm:tree}
	Let $T$ be a tree and $k$ a positive integer. Then, $\gsmb'(T)=k$  if and only if\/ $T$ contains a substructure $S$ from $ \cS_k$, and no substructure from $\cup_{i=1}^{k-1} \cS_i$.
\end{theorem} 
\begin{proof} 
The statement holds for $k=1$ as Staller can win the S-game with her first move if and only if there exists a winning set $N_T[v]$ of cardinality one. It clearly means that $T$ is just an isolated vertex $P_1$ or, equivalently, there is a substructure $P_1 \in \cS_1$ in $T$. From now on, we may proceed by induction on $k$. 
\medskip

First, assume that $T$ contains a substructure $S$ from $ \cS_k$, where $k \ge 2$, but contains no substructure from $\cS_j$ if $j <k$. The latter condition implies, by the induction hypothesis, that $\gsmb'(T) \ge k$. Now consider the substructure $S\in \cS_k$ in $T$ and its origin vertex $z_k$. By the construction of the graphs belonging to $\cS_k$, the forest $T-z_k$ contains one component $T_1$ with a substructure $S^1 \in \cS_{k-1}$ and another component, say $T_2 $, with a substructure  $S^2 \in \cup_{p=1}^{k-1}\cS_p$. Then, $\gsmb'(T_2) \le \gsmb'(T_1) \le k-1$ and Proposition~\ref{prop:cut}(iii) implies
$$\gsmb'(T) \le \gsmb'(T_1)+1 \le k.$$
Therefore, $\gsmb'(T)=k$ follows that completes the proof of one direction of the statement. 
\medskip

Now, assume that $\gsmb'(T)=k$. By the hypothesis, $T$ contains no substructure from $\cup_{p=1}^{k-1} \cS_p$. Let $s_1$ be an optimal first move of Staller in the S-game and $d_1$ an optimal response of Dominator. By Lemma~\ref{lem:game-on-tree}(ii), we may suppose that $d_1$ is a neighbor of $s_1$. Let $T_1, \dots ,T_\ell$ be the components of $T-s_1$ such that $d_1 \in V(T_1)$. After the move $d_1$, by Proposition~\ref{prop:delete-shrink}, the game continues on $\cH'=(\cH_T \mid s_1) -d_1$ with Staller's next optimal move $s_2$. By Lemma~\ref{lem:game-on-tree}(i), $s_2 \notin V(T_1)$. By our supposition, the moves $s_1$ and $d_1$ were optimal and therefore, $\w_M^M(\cH')=k-1$. 

Observe that $\cH'$ consists of several components and  $\ell-1$ of them exactly correspond to the closed neighborhood hypergraphs $\cH_{T_2}, \dots , \cH_{T_\ell}$. If $d_1$ is a leaf in $T$, then there is no further component in $\cH'$. Otherwise, the further components of $\cH'$ are contained in $\cH_1=\cH_{T_1}-d_1$. Proposition~\ref{prop:H-comp} implies that $\w_M^M(\cH')$, that is $k-1$, equals the minimum of the values $\w_M^M(\cH_1), \w_M^M(\cH_{T_2}), \dots,  \w_M^M(\cH_{T_\ell})$. It also follows that Staller's (Maker's) optimal strategy is to play in the same (appropriate) component of $\cH'$ starting with her move $s_2$. As $s_2 \notin V(\cH_1)$, there exists a tree $T_j$, $2 \le j \le \ell$, such that $\gsmb'(T_j)=k-1$. By our hypothesis, $T_j$ contains a substructure $S^1$ from $\cS_{k-1}$. As $S^1$ is a subgraph but not a substructure in $T$, the vertex $s_1$ must be adjacent to a vertex $t_j$ from $X(S^1)$. Note that, as $T$ is a tree, $s_1$ is adjacent to exactly one vertex $t_i$ from each $T_i$, $i\in [\ell]$. Consequently, no component $T_i$ can contain two vertex disjoint substructures from $\cup_{p=1}^{k-1} \cS_p$ and, if $T_i$ contains a substructure $S$ from $\cup_{p=1}^{k-1} \cS_p$, then $t_i$, the neighbor of $s_1$ in $T_i$, belongs to $X(S)$.

We will show that $T-s_1$ contains a substructure $S^2$ from $\cup_{p=1}^{k-1} \cS_p$ that is vertex disjoint to $S^1$. Assume for a contradiction that it is not true. Thus, by the hypothesis, $\gsmb'(T_i) \ge k$ holds for every $i\neq j$. In this situation, suppose that Dominator plays $d_1'=t_j$ instead of $d_1=t_1$. By Lemma~\ref{lem:game-on-tree}(i), each remaining move of Staller belongs to a a subtree $T_i$, $i\neq j$, and she needs at least $k$ further moves to win the game. This contradicts the condition $\gsmb'(T)=k$ and the assumption that Staller plays optimally.  

To conclude the proof, we infer that $T-s_1$ contains two different components, $T_j$ and $T_s$, $j\ne s$, such that an $S^1\in \cS_{k-1}$ is a substructure in $T_j$ and an $S^2\in \cup_{p=1}^{k-1} \cS_p$ is a substructure in $T_s$. Further, $s_1$ is adjacent to $t_j$ and $t_s$ which belong to $X(S^1)$ and  $X(S^2)$ respectively. By the definition of $\cS_k$, the subgraph induced by $V(S^1)\cup V(S^2)\cup \{s_1\}$ is included in $\cS_k$ and, as the subgraph complies with the fixed vertex degrees, it is a substructure from $\cS_k$ in $T$.
\end{proof}

\begin{remark}
	Observation~\ref{obs:S_k}(v) shows that $P_n \in \cS_{\lfloor \log_2 n \rfloor +1}$ for every odd integer $n$. Then, by Theorem~\ref{thm:tree}, we may conclude $\gsmb'(P_n)= \lfloor \log_2 n \rfloor +1$. This provides an alternative shorter proof for the formula in Theorem~\ref{thm:path}.
\end{remark}

For any graph $G$ and a subset $A \subseteq V(G)$, the hypergraph $\cH^{-A}_G$ is obtained from the closed neighhborhood hypergraph $\cH_G$ by deleting the hyperedges corresponding to the closed neighborhoods of vertices in $A$. 

To state an upper bound on the winning numbers of hypergraphs, we define a family $\cF_k$ for every positive integer $k$ as follows:
$$ \cF_k=\{ \cH^{-(V(S) \setminus X(S))}_S: S\in \cS_k \}.
$$
Alternatively, $\cF_k$ can be defined recursively starting with $\cF_1=\{\cB\}$ where $\cB$ is the hypergraph that consists of one vertex $x$ and one edge $\{x\}$. Then, for every $k \ge 2$, the family $\cF_k^*$ contains a hypergraph $\cA$ if $\cA$ can be obtained in the following way. Choose two hypergraphs $\cA_1 \in \cF_{k-1}$ and $\cA_2 \in \cup_{i=1}^{k-1} \cF_i$ and select two edges $e_1 \in E(\cA_1)$ and $e_2 \in E(\cA_2)$. To finish the construction, we take a new (origin) vertex $u$, define $e_i'=e_i \cup \{u\}$ for $i \in [2]$ and set
$$V(\cA)=V(\cA_1) \cup V(\cA_2) \cup \{u\}, \quad 
E(\cA)= E(\cA_1) \cup E(\cA_2) \cup \{e_1', e_2'\} \setminus \{e_1, e_2\}.$$
After having the family $\cF_k^*$ in hand, we set $\cF_k=\cF_k^*\setminus  \cup_{i=1}^{k-1} \cF_i$.

\begin{proposition} \label{prop:F_k}
	If a hypergraph $\cH$ contains a subhypergraph $\cH'$ from $\cF_k$, then $\w_M^M(\cH) \le k$.
\end{proposition}
\begin{proof} The statement is clearly true for $k=1$ as only a presence of a one-element winning set ensures that Maker can win with his first move in a Maker-start game. If $k \ge 2$, we refer to the notation used in the recursive definition of $\cF_k$ and show that Maker has a strategy to win in at most $k$ moves. Consider the subhypergraph $\cH' \in \cF_k$ and suppose that Maker first plays the origin vertex $u$. That leaves the hypergraph  $\cH\mid u$ with a subhypergraph $\cH'\mid u$ for the continuation of the game. $\cH\mid u$ therefore contains two vertex disjoint subhypergraphs $\cA_1 \in \cF_{k-1}$ and $\cA_2 \in \cup_{i=1}^{k-1} \cF_i$. After Breaker's move $v$, at least one of these subhypergraphs remains untouched and $(\cH\mid u)-v$ contains a subhypergraph from $\cup_{i=1}^{k-1} \cF_i$. By the induction hypothesis, Maker can win the game in at most $k-1$ further moves. As follows, $\w_M^M(\cH) \le k$ that concludes the proof.
\end{proof}

The following consequence of Proposition~\ref{prop:F_k} gives an upper bound for the SMBD-numbers of graphs in general. Note that the corollary can also be obtained by analyzing the proof of Theorem~\ref{thm:tree}.

\begin{corollary} \label{cor:F_k}
	If a graph $G$ contains a substructure $S\in \cS_k$, then $\gsmb'(G)\le k$.
\end{corollary} 
\begin{proof} Observe that the closed neighborhood hypergraph $\cH_G$ contains a subhypergraph $\cH'$ isomorphic to $\cH^{-(V(S) \setminus X(S))}_S$. As $S \in \cS_k$, we have $\cH' \in \cF_k$ and Proposition~\ref{prop:F_k} implies  $\w_M^M(\cH_G) \le k$. Since $\gsmb'(G)= \w_M^M(\cH_G)$, this proves the statement.
\end{proof}

In view of Theorem~\ref{thm:tree}, we pose: 

\begin{problem}
	Determine the computational complexity of the decision problem whether $\gsmb'(G)\le k$ holds for a graph $G$ and a positive integer $k$, where $k$ is part of the input. In particular, what about the same question restricted to the class of trees? 
\end{problem}

\section{Subdivided stars}

In this section we consider the SMBD-numbers of subdivided stars. Let $S(n_1,\dots, n_\ell)$ be the tree obtained from the paths $P_{n_1}, \dots, P_{n_\ell}$ and a {\em central vertex} $x$ by making $x$ adjacent to one end of each path. The paths $P_{n_1}, \dots, P_{n_\ell}$ obtained after the deletion of the central vertex are called \emph{branches}. Throughout, we assume that $\ell \ge 2$ and $n_1,\dots, n_\ell$ are positive integers.

For a subdivided star $T=S(n_1,\dots, n_\ell)$, the residual graph $R(T)$, which is obtained by iteratively removing pendant edges, clearly contains at most one strong support vertex. Therefore, by Theorem~\ref{thm:outcometrees}(i), $\gsmb(T)=\infty$ always holds. Concerning the S-game on $T$, we consider three cases according to the parities of $n_1, \dots , n_\ell$. If there is exactly one odd number among  $n_1, \dots , n_\ell$, then $T$ has a perfect matching and $\gsmb'(T)= \infty $ follows from Theorem~\ref{thm:outcometrees}(ii). If there are at least two odd numbers among  $n_1, \dots , n_\ell$, Theorem~\ref{thm:subdiv-odd} will establish the explicit formula. The last case is when $T$ is an \emph{all-even} subdivided star that is, $n_i$ is even for each $i \in [\ell]$. If $\ell = 2$, $T$ is a path and Theorem~\ref{thm:path} establishes the formula for $\gsmb'(T)$. If $\ell=3$, our   Theorem~\ref{thm:subdiv-even} will answer the question. For larger values of $\ell$, we provide a sharp upper bound, but the exact formula for SMBD-numbers of all-even subdivided stars with at least four branches remains an open problem. 

\subsection{Not-all-even subdivided stars} \label{sec:not-all-even}

As Staller cannot win the S-game on $T=S(n_1,\dots, n_\ell)$ if $T$ contains exactly one odd branch, we concentrate on the case when there exist at least two odd branches.

\begin{theorem} \label{thm:subdiv-odd}
	If $S(n_1,\dots, n_\ell)$ is a subdivided star and $n_1$, $n_2$ are the two smallest odd numbers among $n_1, \dots, n_\ell$, then
	$$\gsmb'(S(n_1,\dots, n_\ell))=\lceil \log_2(n_1+n_2+1)\rceil.$$
\end{theorem}
\begin{proof}
If $\ell =2$, the subdivided star is an odd path of order $n_1+n_2+1$ and, by Theorem~\ref{thm:path}, $\gsmb'(S(n_1,n_2))=\lceil \log_2(n_1+n_2+1)\rceil$ is true. From now on we assume $\ell \ge 3$.

Let $T=S(n_1,\dots, n_\ell)$ and $k=\lceil \log_2(n_1+n_2+1)\rceil$. By Theorem~\ref{thm:tree}, it suffices to show that  $T$ contains a substructure   from $ \cS_k$ and contains no substructure from $\cup_{i=1}^{k-1} \cS_i$. The first part of the statement clearly holds because the paths $P_{n_1}$ and $P_{n_2}$ together with the central vertex $x$ form a  subgraph $P_{n_1+n_2+1}$ in $T$. By Observation~\ref{obs:S_k}(v),  $P_{n_1+n_2+1} \in \cS_k$. Further, by Observation~\ref{obs:S_k}(i),  since $n_1$ and $n_2$ are odd numbers, $x$ does not belong to $X(P_{n_1+n_2+1})$. We infer that  $P_{n_1+n_2+1}$ is a substructure in $T$.
\medskip

Assume that $T$ contains a substructure $S \in \cS_i$. By Observation~\ref{obs:S_k}(i) and (ii), $S$ is a tree where $X(S)$ 
corresponds to the partite class containing all leaves of $S$, while all vertices in the partite class $V(S)\setminus X(S)$ are of degree $2$ in $S$. As at least two leaves of $T$ together with the path between them are present in $S$, the central vertex $x$ also belongs to $S$. 
\begin{itemize}
	\item Suppose that $x \in X(S)$. Then $N_T[x] \subseteq V(S)$ and hence $S\cong T$. 
	However, as some odd numbers are present among $n_1, \dots , n_\ell$, Observation~\ref{obs:S_k}(i) implies $x \notin X(S)$  that contradicts the supposition $x \in S$. 
	\item Suppose that $x \notin X(S)$. By Observation~\ref{obs:S_k}(i) and (ii), $\deg_S(x)=2$ and $x$ is at an odd distance from each leaf of $S$. We infer that the substructure $S$ is a path induced by $\{x\} \cup V(P_{n_a}) \cup V(P_{n_b})$, where $n_a$ and $n_b$ are odd integers. By our condition, $n_a+n_b+1 \ge n_1+n_2 +1$ and, consequently, $i \ge k$ whenever $S \in \cS_i$ holds.
\end{itemize} 
By Theorem~\ref{thm:tree} we conclude that $\gsmb'(S(n_1,\dots, n_\ell)) = k = \lceil \log_2(n_1+n_2+1)\rceil$.
\end{proof}

\subsection{All-even subdivided stars} \label{sec:all-even}

Suppose that $T=S(n_1, n_2, n_3)$ is an all-even subdivided star and $n_1 \leq n_2 \leq n_3$. We say that $T$ is \emph{reducible}, if $\lceil \log_2(n_3)\rceil = \lceil \log_2(n_1+n_2+n_3+1)\rceil$. Then we define the \emph{reduced graph} of $T$ as   
$$T'=S(n_1,n_2, n_3-2^{\lceil \log_2 (n_1+n_2+n_3+1) \rceil -1}).$$
By the condition on $n_3$ and as $\lceil \log_2(n_1+n_2+n_3+1)\rceil \ge 3$, $T'$ is also an all-even subdivided star with three branches. 
We say that $T$ is \emph{non-reducible}, if for every $i \in [3]$,
\begin{equation}
	\label{eq:non-reducible}
	\lceil \log_2(n_i)\rceil < \lceil  \log_2(n_1+n_2+n_3+1)\rceil\,.
\end{equation}
Using this terminology, we give a recursive definition for two types of subdivided stars.
\begin{definition} \label{def:type}
	Let $T=S(n_1, n_2, n_3)$ be a subdivided star such that $n_1, n_2, n_3$ are positive even integers and $n_1\leq n_2 \leq n_3$ holds. Let us set $n=n_1+n_2+n_3+1$.
	\begin{itemize}
		\item[$(i)$] If $T$ is non-reducible, then 
		\begin{itemize}
			\item $T \in \cT_1$, if $\lceil \log_2 (n_1+n_2+1) \rceil = \lceil \log_2 n \rceil$;
			\item $T \in \cT_0$, if $\lceil \log_2 (n_1+n_2+1) \rceil < \lceil \log_2 n \rceil$.
		\end{itemize}
		\item[$(ii)$] If $T$ is reducible and $T'$ is its reduced graph, then
		\begin{itemize}
			\item $T \in \cT_1$, if $T' \in \cT_1$ and $\lceil \log_2 (n-2^{\lceil \log_2 n \rceil-2}) \rceil = \lceil \log_2 n \rceil$;
			\item otherwise $T \in \cT_0$.
		\end{itemize}
	\end{itemize}
\end{definition}
By~\eqref{eq:non-reducible} and Definition~\ref{def:type}(i), $T$ is non-reducible and $T\in \cT_1$ if and only if 
\begin{equation}
	\label{eq:non-reducible-T1}
	\lceil \log_2 n_3 \rceil < \lceil \log_2 (n_1+n_2+1) \rceil = \lceil \log_2 (n) \rceil\,.
\end{equation}
Similarly, $T$ is non-reducible and $T\in \cT_0$ if and only if
\begin{equation}
	\label{eq:non-reducible-T0}
	\max\{\lceil \log_2 n_3 \rceil, \lceil \log_2 (n_1+n_2+1)\rceil\} < \lceil \log_2 (n) \rceil\,.
\end{equation}

\begin{theorem} \label{thm:subdiv-even}
	If $T=S(n_1,n_2,n_3)$ is a subdivided star of order $n=n_1+n_2+n_3+1$,  and  $n_1, n_2,n_3$ are positive even integers, then  
	\begin{align} 
		\gsmb'(T) &= \begin{cases} 
			\lceil \log_2 n \rceil; & \enskip T\in \cT_0, \\
			\lceil \log_2 n \rceil +1 ; & \enskip T\in \cT_1. 
		\end{cases}	
	\end{align}
\end{theorem} 
\begin{proof} We use the notation given in the theorem and assume without loss of generality that $n_1\leq n_2 \leq n_3$ holds. Let $x$ be the center of $T$ while $P^1, P^2, P^3$ denote the branches of order $n_1,n_2,n_3$, respectively.

First we show that $T \in \cS$. As $n$ is odd, there is no perfect matching in $T$ and Theorem~\ref{thm:outcometrees}(ii) implies $\gsmb'(T) <\infty$. Thus, by Theorem~\ref{thm:tree}, there exists a substructure $S\in \cS$ in $T$. By Observation~\ref{obs:S_k}(i), $S$ is a tree and the set $X(S)$ contains all leaves of $S$ and the vertices being at an even distance from a leaf. By Observation~\ref{obs:S_k}(iii), every leaf of $S$ is a leaf in $T$ and therefore, $x \in X(S)$. Hence, $N_S(x)=N_T(x)$ and all leaves of $T$ are present in $S$. Consequently, $S=T$ and $T \in \cS$ hold.  

Second, we prove that $\gsmb'(T)$ equals either $\lceil \log_2 n \rceil$ or $\lceil \log_2 n \rceil +1$. As $T\in \cS$, Observation~\ref{obs:S_k}(iv) and Theorem~\ref{thm:tree} immediately give $ n \leq 2^{\gsmb'(T)}-1$. This proves $\lceil \log_2 n \rceil \leq \gsmb'(T)$. For the upper bound, we consider the cut vertex $t_3$ that is the neighbor of $x$ from the branch $P^3$. The graph $T-t_3$ contains two components one of which is a path of order $n_3-1$ and the other one is a path of order $n_1+n_2+1$. Recall that if $m$ is a positive even integer, then $\lceil \log m\rceil = \lfloor \log (m-1)\rfloor + 1$. Using this fact, Proposition~\ref{prop:cut}(iii), Theorem~\ref{thm:path}, and inequalities $n_3 < n$ and $n_1+n_2+1 <n$, we can estimate $\gsmb'(T)$ as follows: 
\begin{align*}
	\gsmb'(T) & \leq \max \{\gsmb'(P_{n_3-1}), \gsmb'(P_{n_1 + n_2 + 1})\} + 1 \\
	& = \max\{\lfloor \log_2 (n_3-1) \rfloor + 1, \lfloor \log_2 (n_1+n_2+1) \rfloor +1 \} +1 \\
	& = \max\{\lceil \log_2 n_3 \rceil, \lceil \log_2 (n_1+n_2+1) \rceil  \} + 1 \\
	& \leq \lceil \log_2 n \rceil +1\,. 
\end{align*}

We have just proved that $\gsmb'(T) \in \{ \lceil \log_2 n \rceil, \lceil \log_2 n \rceil+1  \}$. To decide which of the two possibilities holds and to prove the theorem, we first consider all non-reducible graphs and then proceed by induction on the number of vertices when dealing with reducible graphs. Let $k=\lceil \log_2 n \rceil$. 

\paragraph{Case 1:} $T$ is a non-reducible graph. \\
By the definition of non-reducibility, $\lceil \log_2(n_3)\rceil < k$. 
We consider two cases.
\medskip

\noindent
\emph{Case 1.1:} $T$ is non-reducible and $T\in \cT_0$.\\
By~\eqref{eq:non-reducible-T0},  
$$M=\max\{\lceil \log_2 n_3 \rceil, \lceil \log_2 (n_1+n_2+1) \rceil \} \leq k-1 .$$
Considering the cut vertex $t_3 \in N_T(x) \cap V(P^3)$ again, Proposition~\ref{prop:cut}(iii) and  Theorem~\ref{thm:path} imply $\gsmb'(T) \leq M+1 \leq k$ for this case. Together with the lower bound $k \leq \gsmb'(T)$, this proves $\gsmb'(T)=k$ for every non-reducible $T$ from $\cT_0$. 
\medskip

\noindent
\emph{Case 1.2:} $T$ is non-reducible and $T\in \cT_1$.\\
Since Staller can win the S-game on $T$, there is an optimal first move $s_1$ for her. Note first that $s_1 \neq x$. Indeed, if $s_1 = x$, then Dominator plays any neighbor of $x$ and the undominated vertices can be covered by a matching, hence Dominator will win the game according to Lemma~\ref{lem:pairing}. We may claim the same if all components of $T-s_1$ are even. To see it, observe that Dominator can always play a neighbor of $s_1$ such that the undominated vertices can be covered by a matching. We have thus demonstrated that $s_1\ne x$ and that not all components of $T-s_1$ are even. 

If $s_1$ is from $P^3$, let $T^1$ and $T^2$ be the two components of $T-s_1$ such that $V(T^1) \subseteq V(P^3)$.  Dominator may play the neighbor of $s_1$ that belongs to $T^1$. Then, by Lemma~\ref{lem:game-on-tree}, Staller's optimal response is a vertex from $T^2$ and, under optimal strategies, the game continues on $T^2$ and finishes with the $(\gsmb'(T^2)+1)^{\text{st}}$ move of Staller. As $T^2$ is an all-even subdivided star or an odd path, $T^2 \in \cS$. Then, by~\eqref{eq:non-reducible-T1}, Observation~\ref{obs:S_k}(iv), and by our condition, we get
$$ \gsmb'(T^2)+1 \geq \lceil \log_2 |V(T^2)| \rceil +1 \geq \lceil \log_2 (n_1+n_2+1) \rceil +1 =k+1. $$  
This proves that after Staller's move $s_1 \in V(P^3)$, Dominator has a strategy which ensures that Staller cannot win before her $(k+1)^{\text{st}}$ move. The proof is similar, if $s_1$ is from $P^1$ or $P^2$. We conclude $\gsmb'(T) \geq k+1$, which in turn results in the desired equality $\gsmb'(T)=k+1$.

\paragraph{Case 2:} $T$ is a reducible graph. \\
The reduced graph of $T$ is $T'=S(n_1,n_2, n_3-2^{k-1})$ and $n'=|V(T')|= n- 2^{k-1} \leq 2^{k-1}$. From now on, we proceed by induction. 
\medskip

\noindent
\emph{Case 2.1:} $T$ is reducible and $T\in \cT_0$.\\
Definition~\ref{def:type}(ii) allows two possibilities for $T$ belonging to $\cT_0$. We first show that $\gsmb'(T') \le k-1$ holds in both cases. 
\begin{itemize}
	\item If $T \in \cT_0$ and $T' \in \cT_0$, the induction hypothesis implies $\gsmb'(T') = \lceil \log_2 n' \rceil \leq k-1$.
	\item If $T \in \cT_0$ and $T'\in \cT_1$, then also $\lceil \log_2 (n-2^{k-2}) \rceil \leq \lceil \log_2 n \rceil -1 = k-1 $ must hold which gives $ n-2^{k-2} \leq 2^{k-1}$. This in turn implies 
	$$n'=n-2^{k-1} \leq  2^{k-2}.$$
	From this inequality, we may infer $\lceil \log_2 n' \rceil +1 \leq k-1$ and conclude $\gsmb'(T') \leq k-1$ again.
\end{itemize}
Let $w$ be the (unique) vertex of $P^3$ such that $T-w$ consists of a component isomorphic to $T'$ and a path component of order $2^{k-1}-1$. We have already seen that $\gsmb'(T') \leq k-1$ and the path component clearly satisfies $\gsmb'=k-1$. Applying Proposition~\ref{prop:cut}(iii), we get
$$\gsmb'(T) \leq \max\{\gsmb'(T'), k-1 \} +1 \leq (k-1) +1 =k . $$ 
Since $\gsmb'(T) \ge k$ also holds, we conclude $\gsmb'(T)=k$ as stated in the theorem.
\medskip

\noindent
\emph{Case 2.2:} $T$ is reducible and $T\in \cT_1$.\\
We are going to prove that  $\gsmb'(T) = k+1$. As  $\gsmb'(T) \leq k+1$ always holds, it suffices to show that there is no optimal first move $s_1$ for Staller that allows her to win within $k$ (or less) moves. 
\begin{itemize}
\item If $T-s_1$ contains no component of odd order, then Staller cannot win in the continuation of the game. 
\item If Staller plays a vertex $s_1 \in V(P^1) \cup V(P^2)$, then Dominator responds by choosing the neighbor of $s_1$ which is not incident to the $(s_1,x)$-path. Then, by Lemma~\ref{lem:game-on-tree}, Staller's optimal strategy is playing on the component $T^2$ of $T-s_1$ which contains $x$. As $T^2 \in \cS$ and $|V(T^2)| \geq n_3 +n_1 +1 \geq 2^{k-1} +3$, Observation~\ref{obs:S_k}(iv) implies $\gsmb'(T^2) \geq k$. This shows that if Staller chooses her first vertex from $P^1$ or $P^2$, she cannot win with $k$ moves. 
\item If $s_1 \in V(P^3)$ such that $T-s_1$ contains a path component of order at least $2^{k-1}+1$, then we can handle it as in the previous paragraph. 
	\item Let $w$ be the vertex from $P^3$ such that $T-w$ contains a component isomorphic to  $T'$ and a path of order $2^{k-1}-1$. If $s_1=w$, Dominator may reply by playing the neighbor of $w$ from the path and then, by Lemma~\ref{lem:game-on-tree}, Staller's optimal strategy is to play (optimally) on $T'$ in the continuation. Hence, she needs at least $\gsmb'(T')+1$ moves to win the game. By Definition~\ref{def:type}, $T \in \cT_1$ only if $T'\in \cT_1$ and $n-2^{k-2}>2^{k-1}$. It follows that
	$$n'=n-2^{k-1} > 2^{k-2}\,,$$
	and hence $\lceil \log_2 n'\rceil \ge k-1$. 
	By hypothesis,  $T'\in \cT_1$ implies  $\gsmb'(T')= \lceil \log_2 n' \rceil +1$ and therefore, we have  $\gsmb'(T') \geq k$ by the above inequality. We may conclude again that Staller needs at least $k+1$ moves to win the game. 
	\item  The last case is when $s_1$ is a vertex from $P^3$ the removal of which results in a path component of order $2^{k-1}-1-2a$ (for a positive integer $a$) and an all-even subdivided star $T''$. In this case, again, Dominator can play the neighbor of $s_1$ from the path and Staller needs at least further $\gsmb'(T'')$ moves to win. Observe that $T'$ can be obtained from $T''$ by iteratively ($a$ times) removing  weak support vertices of degree $2$ and the attached leaves. It follows then from Proposition~\ref{prop:cut}(i) that $\gsmb'(T')\leq \gsmb'(T'')$. We conclude that Staller needs at least $\gsmb'(T'') +1 \geq \gsmb'(T') +1 \geq k+1$ moves to win. 
\end{itemize}
The discussed cases show that $\gsmb'(T) \geq k+1$, if $T$ is reducible and contained in $\cT_1$. It settles the last inductive case.  
\end{proof}

\begin{corollary}
	If $k\ge 1$ is an integer, then 
	$$\gsmb'(S(2k,2k,2k)) = \lceil \log_2(4k+1)\rceil + 1\,.$$
\end{corollary}

\begin{proof}
Let $k\ge 1$ and set $T = S(2k,2k,2k)$. Then $T$ is non-reducible because $\lceil \log_2(2k)\rceil < \lceil \log_2(6k+1)\rceil$. Assume first that $T\in \cT_1$. Then, by definition, $\lceil \log_2(4k+1)\rceil = \lceil \log_2(6k+1)\rceil$. Moreover, by Theorem~\ref{thm:subdiv-even}, 
$$\gsmb'(T) = \lceil \log_2(6k+1)\rceil + 1 = \lceil \log_2(4k+1)\rceil + 1\,.$$ 
Assume second that $T\in \cT_0$. By definition, $\lceil \log_2(4k+1)\rceil < \lceil \log_2(6k+1)\rceil$. Since $|\log_2(6k+1) - \log_2(4k+1)| < 1$, it follows that $\lceil \log_2(6k+1)\rceil - \lceil \log_2(4k+1)\rceil \le 1$. Consequently, $\lceil \log_2(4k+1)\rceil + 1 = \lceil \log_2(6k+1)\rceil$. Applying Theorem~\ref{thm:subdiv-even} again, we get that also in this case 
$$\gsmb'(T) = \lceil \log_2(6k+1)\rceil = \lceil \log_2(4k+1)\rceil + 1\,,$$ 
and we are done. 
\end{proof}

The missing case for subdivided stars is the determination of $\gsmb'(S(n_1, \dots, n_\ell))$, where $\ell \ge 4$ and each $n_i$ is even. We leave it as an open problem, but provide a sharp upper bound on the parameter.

\begin{problem}
	Determine the value of $\gsmb'(S(n_1, \dots, n_\ell))$ for the cases when $\ell \geq 4$ and  $n_i$ is a positive even number for each $i \in [\ell ]$.
\end{problem}
\medskip

\begin{proposition}
	Let $T=S(n_1, \dots, n_\ell)$ be a subdivided star such that $\ell \ge 3$, $n_1\geq \dots \geq n_\ell$, and $n_i$ is a positive even integer for each $i \in [\ell ]$. Then,
	$$\gsmb'(T) \leq \max(\{i+ \lceil \log_2 n_i \rceil: i \in [\ell-2]  \} \cup \{\ell-2 + \lceil \log_2(n_{\ell-1}+n_\ell +1)\rceil \} ), $$
	and the bound is sharp for every $\ell \ge 3$.
\end{proposition}
\begin{proof}  Let $x$ be the center of $T$, and $P^1, \dots , P^\ell$ be the branches of order $n_1, \dots ,n_\ell$ respectively. The neighbor of $x$ from $P^i$ will be denoted by $t_i$ for $i\in [\ell]$. 

To prove the upper bound, we describe a strategy for Staller. In her first move, Staller plays $t_1$. We have two cases.
\begin{itemize}
	\item If Dominator replies with a move $d_1$ from the component of $T-t_1$ which contains $x$, Staller continues the game by playing (optimally) on the path $P^1-t_1$. After the moves $t_1$, $d_1$, the Maker-Breaker game continues on the hypergraph $\cH'= (\cH_T -t_1) \mid d_1$ with Staller's (Maker's) move. As $\cH''= \cH_{P^1-t_1}$ is a subhypergraph in $\cH'$, Proposition~\ref{prop:delete}(i) implies 
	$$\w_M^M(\cH')\leq \w_M^M(\cH'')=\gsmb'(P^1-t_1)= \lceil \log_2 n_1 \rceil .$$
	In this case, Staller can win the S-game on $T$ in $1+ \lceil \log_2 n_1 \rceil $ moves.
	\item If Dominator replies with a move $d_1 \in V(P^1-t_1)$, then Staller plays $t_2$ as her next move.  
\end{itemize}

In general, Staller's strategy is the following. For $i \in [\ell-2]$, if Staller plays $t_i \in N_T(x) \cap V(P^i)$ as her $i^{\text{th}}$ 
move and Dominator responds by playing a vertex $d_i \notin V(P^i)$, then Staller continues playing (optimally) on $P^i -t_i$. It allows her to win the game in $i +  \lceil \log_2 n_i \rceil$ moves. If Dominator responds by choosing a vertex $d_i \in V(P^i)$ and $i \leq \ell-3$, Staller's next move is $t_{i+1}$. If Dominator's reply is $d_i \in V(P^i)$ and $i=\ell-2$, Staller plays on the path $P'$ induced by $V(P^{\ell-1}) \cup \{x\} \cup V(P^{\ell})$ and can win the game in $\ell-2+ \lceil \log_2(n_{\ell-1}+n_\ell +1)\rceil$ moves. 

Consequently, either there exists an $i \in [\ell-2]$ such that Staller plays $t_{i}$ and Dominator replies with $d_{i}\notin V(P^{i})$ or Staller, after playing $t_1, \dots, t_{\ell-2}$, can win by playing on $P'$. In the former case, Staller plays at most $i+ \lceil \log_2 n_i \rceil $ vertices; in the latter case she can win in at most $\ell-2+ \lceil \log_2(n_{\ell-1}+n_\ell +1)\rceil$ moves. It verifies the upper bound in the statement. 
\medskip

To prove the sharpness, we define the following infinite class of all-even subdivided stars. If $\ell\ge 3$ and $p \ge 2$ are two  integers, let  $Z(\ell, p)$ be the all-even subdivided star $S(n_1, \dots n_\ell)$ where $n_\ell=2$, $n_{\ell-1}=2^p-2$, and
$$n_i= \sum_{j=i+1}^{\ell} n_j = 2^{\; \ell+p-i-2} \qquad \mbox{for every $i \in [\ell-2]$.}
$$
Note that $Z(\ell, p)$ contains $n=2^{\ell+p-2}+1$ vertices. As it is an all-even subdivided star, $Z(\ell, p) \in \cS$. By Observation~\ref{obs:S_k}(iv), 
$$\gsmb'(Z(\ell,p)) \geq \lceil \log_2 n \rceil =  \ell+p-1.$$
By determining the upper bound from the current proposition, we find that 
$$i + \lceil \log_2 n_i \rceil = i + (\ell + p - i - 2) = \ell + p - 2$$
for all $i\in [\ell -2]$, and  
$$\ell-2 + \lceil \log_2(n_{\ell-1}+n_\ell +1)\rceil = \ell-2 + p+1 = \ell+p-1\,.$$
Thus the upper bound in the proposition gives $\gsmb'(Z(\ell,p)) \le \ell + p -1$. 
As the upper bound matches the lower bound, the sharpness is true. \end{proof}

\section{Caterpillars}
\label{sec:caterpillars}

A \textit{caterpillar} is a tree $T$ on at least three vertices in which a single path is incident to (or contains) every edge. There is more than one path with this property, but we select the shortest one and call it the \textit{spine} of the caterpillar, see \cite{west_introduction_2000}. The vertices of the spine will be denoted by $v_1,\dots , v_\ell$ according to their natural order. Before stating the main result of the section, some preparation is needed. 

Let $T$ be a caterpillar. Then we set 
$$\po(T) = \{P:\ P\ \mbox{is a maximal path in}\ T\ \mbox{and}\ |V(P)|\ \mbox{is odd}\}\,,$$
where maximal is meant maximal w.r.t.\ inclusion.
We say that a path $P = u_1u_2\dots u_k$ in $T$ is {\em clean} if $\deg_T(u_i) = 2$ for $3\le i\le k-2$. Let further 
$$\pco(T) = \{P:\ P\in \po(T)\ \mbox{and}\ P\ \mbox{is clean}\}\,.$$
We refer to the members of $\pco(T)$ as $\pco$-paths. Note that each such path contains at most two support vertices from the spine of $T$. If $\cP$ is a collection of paths of $T$, then  the smallest carnality of a set of vertices $U$ such that each path from $\cP$ has a vertex in $U$ is denoted by $t_{\rm p}(\cP)$. If $T$ is a caterpillar containing at least one clean path, then let $p(T)$ be the minimum order of such a path. If $t_{\rm p}(\pco(T)) \ge 2$, then let $p^*(T)$ be the smallest integer $s$  such that not all clean paths of order at most $s$ share a vertex, that is, 
$$p^*(T) = \min \{s:\ t_{\rm p}(\{P:\ P\in \pco(T), |V(P)|\le s\}) \ge 2\}\,.$$

\medskip
For the illustration of the above concepts and notation see Fig.~\ref{fig:caterpillar}. For the caterpillar $T$ from the figure we have 
$$\po(T) = \{u_1v_1v_2v_3u_3, u_1v_1v_2v_3u_3', u_3v_3u_3', u_4v_4v_5v_6v_7v_8u_8\}\,.$$ 
Clearly, $p(T) = 3$. Next, $t_{\rm p}(\pco(T)) = 2$ because each path from $\pco(T)$ has a vertex in $\{v_3,v_4\}$ and there is no such set with a single vertex. Finally, $p^*(T) = 7$ because each clean path of order smaller than $7$ contains the vertex $v_3$.

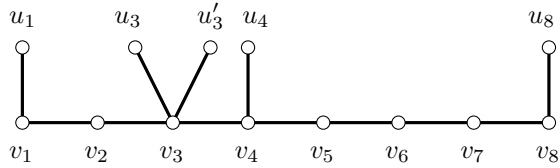
\begin{figure}[ht!]
	\begin{center}
		\begin{tikzpicture}[line cap=round,line join=round,>=triangle 45,x=1.0cm,y=1.0cm]
			\clip(0,1) rectangle (13,4);
			\draw [line width=1.2pt] (3.,2.)-- (4.,2.);
			\draw [line width=1.2pt] (4.,2.)-- (5.,2.);
			\draw [line width=1.2pt] (5.,2.)-- (6.,2.);
			\draw [line width=1.2pt] (6.,2.)-- (7.,2.);
			\draw [line width=1.2pt] (7.,2.)-- (8.,2.);
			\draw [line width=1.2pt] (8.,2.)-- (9.,2.);
			\draw [line width=1.2pt] (9.,2.)-- (10.,2.);
			\draw [line width=1.2pt] (3.,2.)-- (3.,3.);
			\draw [line width=1.2pt] (10.,2.)-- (10.,3.);
			\draw [line width=1.2pt] (5.,2.)-- (4.5,3);
			\draw [line width=1.2pt] (5.,2.)-- (5.5,3);
			\draw [line width=1.2pt] (6.,2.)-- (6.,3.);
			\draw (2.7,3.6) node[anchor=north west] {$u_1$};
			\draw (4.1,3.6) node[anchor=north west] {$u_3$};
			\draw (5.2,3.72) node[anchor=north west] {$u_3'$};
			\draw (5.8,3.6) node[anchor=north west] {$u_4$};
			\draw (9.6,3.6) node[anchor=north west] {$u_8$};
			\draw (2.7,1.8) node[anchor=north west] {$v_1$};
			\draw (3.7,1.8) node[anchor=north west] {$v_2$};
			\draw (4.7,1.8) node[anchor=north west] {$v_3$};
			\draw (5.7,1.8) node[anchor=north west] {$v_4$};
			\draw (6.7,1.8) node[anchor=north west] {$v_5$};
			\draw (7.7,1.8) node[anchor=north west] {$v_6$};
			\draw (8.7,1.8) node[anchor=north west] {$v_7$};
			\draw (9.7,1.8) node[anchor=north west] {$v_8$};
			\begin{scriptsize}
				\draw [fill=white] (3.,2.) circle (2.5pt);
				\draw [fill=white] (4.,2.) circle (2.5pt);
				\draw [fill=white] (5.,2.) circle (2.5pt);
				\draw [fill=white] (6.,2.) circle (2.5pt);
				\draw [fill=white] (7.,2.) circle (2.5pt);
				\draw [fill=white] (8.,2.) circle (2.5pt);
				\draw [fill=white] (9.,2.) circle (2.5pt);
				\draw [fill=white] (10.,2.) circle (2.5pt);
				\draw [fill=white] (3.,3.) circle (2.5pt);
				\draw [fill=white] (10.,3.) circle (2.5pt);
				\draw [fill=white] (4.5,3.) circle (2.5pt);
				\draw [fill=white] (5.5,3) circle (2.5pt);
				\draw [fill=white] (6.,3.) circle (2.5pt);
			\end{scriptsize}
		\end{tikzpicture}
		\caption{Caterpillar $T$.}
		\label{fig:caterpillar}
	\end{center}
\end{figure}

\begin{theorem} \label{thm:caterpillar}
	If $T$ is a caterpillar, then
	\begin{align*}
		\gsmb'(T) &= \begin{cases} 
			\infty; & \enskip \pco(T) = \emptyset, \\
			\lceil \log_2 p(T) \rceil; & \enskip \text{otherwise. }
		\end{cases}	\\
		\gsmb(T) & = \begin{cases} 
			\infty; & \enskip   t_{\rm p}(\pco(T)) \leq 1, \\
			\lceil \log_2 p^*(T) \rceil; & \enskip \text{otherwise. } 
		\end{cases}
	\end{align*}
\end{theorem}

\begin{proof} Let $T$ be a caterpillar and set $p = p(T)$ and $p^* = p^*(T)$. 
	
	Suppose first that $\pco(T)=\emptyset$, which also implies that there are no strong support vertices in $T$. We can construct a perfect matching in $T$ as follows. First, delete the support vertices of $T$ from the subgraph induced by the spine of $T$. As $\pco(T)=\emptyset$, each component of the obtained linear forest $T^-$ is an even path. Therefore, $T^-$ admits a perfect matching. For the remaining vertices of $T$, we may match every leaf with its support vertex. As there are no strong support vertices in $T$  we obtain a perfect matching in $T$. Therefore, we may conclude by Theorem~\ref{thm:outcometrees} that $\gsmb'(T)=\gsmb(T)=\infty$. In the rest of the proof, let $\pco(T)\neq \emptyset$.
	
We now prove the statement for $\gsmb'(T)$.
Since $P_1$ is not a caterpillar, $\gsmb'(T)\ge 2$ holds.
Let $k= \lceil \log_2 p \rceil$. 
By Theorem~\ref{thm:tree}, $\gsmb'(T)=2$ if and only if $T$ contains a strong support vertex, that is, $p=3$. Hence, our statement holds for $p=3$. From now on, we may suppose that $p \ge 5$ which means that there is no strong support vertex in $T$.

Assume that $P_p: xv_i\dots v_{i+p-3}y$ is a shortest path in $\pco(T)$.
It follows from Theorem~\ref{thm:tree} and Observation~\ref{obs:S_k}(v) that $P_p \in \cS_k$ is a substructure in $T$, as the vertices $v_i$ and $v_{i+p-3}$, which may have higher degree in $T$ than in $P_p$, do not belong to $X(P_p)$. By Theorem~\ref{thm:tree}, we may infer $\gsmb'(T) \le k$.

Now, suppose for a contradiction that $T$ contains a substructure $S \in \cS_\ell$ so that $\ell \le k-1$.
Observe first that $S$ contains no vertex $u$ with $\deg_S(u) \ge 3$. Indeed, as $T$ is a caterpillar, such a vertex $u$ would be a support vertex in both $T$ and $S$ by Observation~\ref{obs:S_k}(iii). Then, $\deg_S(u) >2$ contradicts Observation~\ref{obs:S_k}(ii). Therefore, $S$ is an odd path. 
We next show that $S$ cannot be a clean path in $T$. Indeed, $S\in \pco(T)$ implies $|V(S)| \ge p$ and $\ell \ge k$ that is a contradiction. Therefore, $S: av_q \dots v_{q+r}b$ is an odd path between the leaves $a$ and $b$ of $T$ and, for an integer $t\in[r-1]$, the vertex $v_{q+t}$ is adjacent to a leaf $c$ in $T$. By choosing the smallest positive $t$ with this property, we may further assume that $S': av_q \dots v_{q+t}c$ is a clean path in $T$. If  $t$ is odd then, by Observation~\ref{obs:S_k}(i), $v_{q+t} $ belongs to $X(S)$ that contradicts $\deg_S(v_{q+t})=2 < \deg_T(v_{q+t})=3$. If $t$ is even, then $S'$ is an odd clean path and hence, $|V(S')|=t+3 \ge n_{\min}(\pco(T))= p$ that implies $|V(S)|\ge p+2$ and contradicts $\ell \le k-1$. This contradiction, together with Theorem~\ref{thm:tree}, imply $\gsmb'(T)\ge k$ and in turn, $\gsmb'(T)= k$ follows.
\medskip

From now on, suppose that a D-game is played on $T$ and $\pco(T)\neq \emptyset$. If $t_{\rm p}(\pco(T)) = 1$, a vertex contained in all $\pco$-paths of $T$ can be chosen as a support vertex $v_i$. (Indeed, a set of intersecting clean paths must have a common endvertex and a common support vertex.) Let $U$ be the set of leaves adjacent to $v_i$. If Dominator plays $v_i$ as his first move, we may construct a matching satisfying the conditions of Lemma~\ref{lem:pairing}, where $X = \{v_i\}$ and $Y = \emptyset$. First, delete the support vertices of $T$ from the subgraph induced by the spine of $T$ to obtain the linear forest $T^-$. Since every $\pco$-path is incident to $v_i$, each component of $T^-$ that contains neither $v_{i-1}$ nor $v_{i+1}$ is an even path and admits a perfect matching. If there is an odd path component which contains $v_{i-1}$, then $v_{i-1}$ is the end vertex, and we can take a matching that covers all vertices of the component except $v_{i-1}$. The situation is similar if an odd path contains $v_{i+1}$. To finish the construction of the matching $M$, we add every edge which is incident to a leaf except those incident to $v_i$ and $U$. By the condition $t_{\rm p}(\pco(T)) = 1$, the caterpillar contains no strong support vertex except (possibly) $v_i$. This shows that $M$ is a matching in $T-v_i$. Concerning the set of vertices covered by $M$, we observe
$$V(T) \setminus V(M) \subseteq N_T[v_i]\,,$$
and conclude by Lemma~\ref{lem:pairing} that Staller cannot win the D-game if Dominator plays $v_i$ as his first move. Thus, $\gsmb(T)=\infty$. 

For the rest of the proof we assume $ t_{\rm p}(\pco(T)) \ge 2$.
By the definition of $p^*$, after Dominator's first move $d_1$, there remains an unplayed $\pco$-path $P_j$ of order $j \le p^*$. Clearly, $P_j \in \cS_i$ with $i \le \lceil \log_2 p^*\rceil$. By Proposition~\ref{prop:delete-shrink}, the game continues as a Maker-start game on $\cH_T-d_1$. As $d_1$ might be adjacent to a support vertex of $P_j$ but no vertex from $X(P_j)$, the hypergraph $\cH_T-d_1$ contains a subhypergraph $\cH'= \cH_{P_j}^{-(V(P_j)\setminus X(P_j))}$. By Proposition~\ref{prop:F_k},  $\cH' \in \cF_i$ implies $\w_M^M(\cH_T-d_1) \le i$. Assuming that $d_1$ is an optimal first move of Dominator, we conclude 
$$\gsmb(T)= \w_M^M(\cH_T-d_1) \le i \le \lceil \log_2 p^* \rceil.$$

We now prove $\gsmb(T) \ge \lceil \log_2 p^* \rceil$ by describing a strategy of Dominator which ensures that Staller cannot win by playing less than $\lceil \log_2 p^* \rceil$ vertices. Let $v_1\dots v_\ell$ be the spine of the caterpillar. 
According to the definition of $p^*$, there exists a support vertex $v_q$ in $T$ such that the smallest order of an odd clean path not covered by $v_q$ is  $p^*$. Let $v_q$ be the first move of Dominator. By Proposition~\ref{prop:delete-shrink}, the game continues as a Maker-start game on $\cH_T-v_q$ and  $\gsmb(T) \ge \w_M^M(\cH_T-v_q)$. We will transform $\cH_T -v_q$ into a hypergraph $\cH'$ such that 
$\w_M^M(\cH') \le \w_M^M(\cH_T-v_q)$.
If $q \ge 2$, then let $P^1$ be the clean maximal path in $T$ which contains both $v_{q-1}$ and $v_q$. 
\begin{itemize}
	\item[$(a)$] 
	If $P^1$ is an even path, we obtain $\cH'$ from $\cH_T-v_q$ by adding the new winning set (i.e., the edge) $N_T[v_{q-1}] \setminus \{v_q\}$. Then, this component of $\cH'$ corresponds to the closed neighborhood hypergraph of the caterpillar $T_1$ that is induced by the vertices $v_1, \dots , v_{q-1}$ and the adjacent leaves in $T$. If $q=2$, then $T_1$ is a star $K_{1,m}$ with $m \ge 1$.
	\item[$(b)$] If $P^1$ is an odd path,  we obtain $\cH'$ from $\cH_T-v_q$ by replacing the winning set $N_T[v_{q-2}]$ with the smaller hyperedge  $N_T[v_{q-2}] \setminus \{v_{q-1}\}$. This component of $\cH'$ therefore corresponds to the closed neighborhood hypergraph of the caterpillar $T_1$ that is induced by the vertices $v_1, \dots , v_{q-2}$ and the adjacent leaves in $T$. If $P^1$ is odd, then $q=2$ is not possible. If $q=3$, then $T_1$ is a star which contains exactly one vertex, namely  $v_{1}$, from the spine of $T$.
\end{itemize}
If $q\le \ell-1$, we also consider a clean maximal path $P^2$ in $T$ that contains $v_q$ and $v_{q+1}$. Similarly to the changes described above for $P^1$, if $|V(P^2)|$ is even, we add the new winning set $N_T[v_{q+1}] \setminus \{v_q\}$; if  $|V(P^2)|$ is odd, we replace  $N_T[v_{q+2}]$ with   $N_T[v_{q+2}] \setminus \{v_{q+1}\}$ when construct $\cH'$. The corresponding component of $\cH'$ is the closed neighborhood hypergraph of a caterpillar $T_2$. In the first case, $T_2$ is the subgraph of $T$ induced by $v_{q+1}, \dots , v_\ell$ and the adjacent leaves, while in the second case $T_2$ is induced by $v_{q+2}, \dots , v_\ell$ and the adjacent leaves. Note that $T_1$ is considered empty if $q=1$, while $T_2$ is empty if $q=\ell$. In these cases we respectively set $\gsmb'(T_1)=\infty$ and $\gsmb'(T_2)=\infty$.

Observe that  $T_1$ and $T_2$ still contain all clean maximal paths of $T$ that do not contain $v_q$. In particular, a $\pco$-path of order $p^*$ is present in the disjoint union $T'=T_1 \cup T_2$. In both cases $(a)$ and $(b)$, only clean maximal paths of even order were added as new ones. Note that, if $|V(P^{i})| \le 5$, no new clean maximal path appears in $T_i$, for $i \in [2]$. Using Propositions~\ref{prop:delete} and \ref{prop:comp}, and the formula already proved for the S-game on a caterpillar, we obtain the following relation:
$$\gsmb(T) \ge \w_M^M(\cH_T-v_q) \ge \w_M^M(\cH')=  \min\{\gsmb'(T_1), \gsmb'(T_2)\} =\lceil \log_2 p^* \rceil.$$
This completes the proof for the second formula.
\end{proof}



\nocite{*}
\bibliographystyle{abbrvnat}
\bibliography{MBD-game-dmtcs}
\label{sec:biblio}

\end{document}